\documentclass[11pt]{article}

\usepackage{amsmath,mathrsfs}
\usepackage{amsfonts}
\usepackage{graphicx}
\usepackage{amsthm}
\usepackage[utf8]{inputenc}
\usepackage{graphicx}
\usepackage{float}
\usepackage{bbm}
\usepackage{bm}
\usepackage{setspace}
\usepackage{hyperref}
\usepackage{cite}

\usepackage[a4paper]{geometry}
\usepackage{times}

\allowdisplaybreaks[3]

\newcommand{\E}{\mathbb{E}}
\newcommand{\Q}{\mathbb{Q}}

\newcommand{\D}{\mathrm{d}}
\DeclareMathOperator{\Tr}{Tr}

\newtheorem{lemma}{Lemma}[]
\newtheorem{corollary}{Corollary}[]
\newtheorem{theorem}{Theorem}[]
\theoremstyle{definition}
\newtheorem{remark}{Remark}[]

\title{Particle approximation for a conditional McKean--Vlasov stochastic differential equation}

\author{
Kai Du\thanks{Shanghai Center for Mathematical Sciences, Fudan University, Shanghai 200438, China;
Shanghai Artificial Intelligence Laboratory, Shanghai, China. The author is a member of LMNS, Fudan University. 
{\small\it E-mail:} {\small\tt kdu@fudan.edu.cn}. The author's research is supported by the National Science and Technology Major Project (2022ZD0116401)
and by the National Natural Science Foundation of China (No.~12222103).}
\and
Yunzhang Li\thanks{Corresponding Author. Research Institute of Intelligent Complex Systems, Fudan University, Shanghai 200433, China. {\small\it E-mail:} {\small\tt li\_yunzhang@fudan.edu.cn}. Research supported by the National Natural Science Foundation of China (No.~12301566), and by the Science and Technology Commission of Shanghai Municipality (No.~23JC1400300), and by the Chenguang Program of Shanghai Education Development Foundation and Shanghai Municipal Education Commission (No.~22CGA01).}
\and
Yuyang Ye\thanks{Department of Finance and Control Sciences, School of Mathematical Sciences, Fudan University, Shanghai 200433, China. {\small\it E-mail:} {\small\tt 22110180052@fudan.edu.cn}.}
}
\date{}

\def\be{\begin{eqnarray}}
\def\bel{\begin{equation}\label}
\def\ee{\end{eqnarray}}

\numberwithin{equation}{section}
\numberwithin{theorem}{section}
\numberwithin{lemma}{section}
\numberwithin{corollary}{section}
\numberwithin{remark}{section}

\allowdisplaybreaks[4]

\begin{document}
\maketitle

\begin{abstract}
In this paper, we construct a type of interacting particle systems to approximate a class of stochastic different equations whose coefficients depend on the conditional probability distributions of the processes given partial observations.
After proving the well-posedness and regularity of the particle systems, 
we establish a quantitative convergence result for the empirical measures of the particle systems in the Wasserstein space, as the number of particles increases. 
Moreover, we discuss an Euler--Maruyama scheme of the particle system and validate its strong convergence. 
A numerical experiment is conducted to illustrate our results.

\bigskip


\medskip

\end{abstract}

\section{Introduction}

This paper is concerned with a conditional McKean--Vlasov stochastic differential equation (CMVSDE) in a completed probability space $(\Omega, \mathcal{F}, \mathbb{P})$:
\begin{equation}
\left\{
\begin{aligned}\label{partial}
X_t&= x + \int_0^t b(X_s,\pi_s)\,\D s+  \int_0^t \sigma_1  (X_s, \pi_s )\,\D W_s +  \int_0^t \sigma_2 (X_s, \pi_s)\,\D B_s ,
\\[0.3cm]
Y_t&=  \int_0^t h(X_s,\pi_s)\,\D s+B_t,
\qquad\quad  \text{with }\ 
\pi_s(\cdot) = \mathbb{P} \big[{X}_s\in \cdot \,|\mathcal{F}_s^Y \big],
\end{aligned} 
\right.
\end{equation}
where $x\in\mathbb{R}^n$, $\mathcal{F}_t^Y=\sigma\{Y_s:s\leq t\}$, and $(B,W)$ is a standard Brownian motion.  
The process $Y$ serves as observations, resulting in a conditional law $\pi$ of $X$ that is involved in the dynamics of $X$.
We remark that when $h$ vanishes, the equation~\eqref{partial} becomes the so-called McKean--Vlasov SDE with common noise which has been extensively studied in the literature, see e.g. \cite{Crisan, Kumar, Carmona2}, and \cite{Hammersley} for some recent progress. 
As $h$ emerges and relies on $X$, the increased reliance of the observation process $Y$ on the state $X$ substantially amplifies the nonlinearity within the dynamics.
Very recently, the well-posedness of~\eqref{partial} is obtained by Buckdahn, Li, and Ma~\cite{Buckdahn} by employing Banach's fixed point theorem and a localization technique. 
The study on~\eqref{partial} is motivated by applications in mathematical finance and stochastic control.
For instance, Ma, Sun, and Zhou \cite{Ma} formulated a linear form of CMVSDE~\eqref{partial} based on the Kyle--Back equilibrium model;
Buckdahn, Li, and Ma \cite{MR3719957} studied a control problem arising in the mean-field game, in which the coefficients of the controlled dynamics linearly depend on the conditional law of the state.

In this paper, we aim to construct a particle system to approximate CMVSDE~\eqref{partial}. 
To approximate the conditional law $\pi$ in an appropriate way,
we draw upon the methodology from nonlinear filtering theory and introduce the reference probability measure $\mathbb{Q}$ by the Girsanov transformation:
$$
\frac{\D \mathbb{Q}}{\D \mathbb{P}}\Big\vert_{t}=\frac{1}{L_t}:=\exp\bigg(-\int_0^t h(X_s,\pi_s) \D Y_s+\frac{1}{2}\int_0^t \left|h(X_s,\pi_s)\right|^2 \D s\bigg).
$$
As a result, the observation process $Y$ is a standard Brownian motion under $\mathbb{Q}$,
and by Bayes’ formula (or the Kallianpur--Striebel formula), one can write the conditional law $\pi$ defined in~\eqref{partial} as
\begin{equation}\label{2312261}
\pi_s(\cdot) = \mathbb{P} [{X}_s\in \cdot \,|\mathcal{F}_s^Y ] = \frac{\E^{\Q}[{L}_s \bm{1}_{{X}_s\in \,\cdot}\,|\mathcal{F}_s^Y ]}{\E^{\Q} [{L}_s|\mathcal{F}_s^Y ]}.
\end{equation}
Based on this transformation, we introduce a particle system $\{ (X^{i,N},L^{i,N})\}_{i=1}^N$ on $(\Omega, \Q, \mathcal{F})$: 
\begin{equation}\label{particle}
\left\{
\begin{aligned}
X_t^{i,N} & = x + \int_0^t ( b - \sigma_2 h ) (X_s^{i,N},\pi_s^N)\D s+ \int_0^t \sigma_1 (X_s^{i,N},\pi_s^N)\D W_s^i + \int_0^t \sigma_2 (X_s^{i,N},\pi_s^N)\D Y_s, \\[0.3cm]
L_t^{i,N} & = 1 + \int_0^t h (X_s^{i,N},\pi_s^N ) L_s^{i,N} \D Y_s,\\[0.3cm]
\pi_s^N & = \frac{\sum_{i=1}^N {L}^{i,N}_s \delta_{X_s^{i,N}}}{\sum_{i=1}^N {L}^{i,N}_s} = \sum\limits_{i=1}^N {w}_s^{i} \, \delta_{{X}_s^{i,N}} \quad\text{with}\quad {w}^i_s=\frac{{L}^{i,N}_s}{\sum_{i=1}^N {L}^{i,N}_s},
\end{aligned}
\right.
\end{equation}
where $\{W^i\}_{i=1}^N$ are independent Brownian motions under $\Q$ and independent of $Y$. 
Intuitively, the process $L^{i}$ plays a role of a likelihood weight assigned to the path $X^i$ that represents the probability of that path being sampled from the conditional law.
We can illustrate the validity of this model from two perspectives.
First, if the coefficient $h$ is a constant, then $\pi_t^N=\frac{1}{N}\sum_{i=1}^N\delta_{X_t^{i,N}}$ is an unweighted empirical measure and converges, as $N\to\infty$, to the conditional law of a McKean--Vlasov process with common noise (cf.~\cite{Crisan}).
Second, if the coefficients are all independent of the distribution $\pi$, then~\eqref{particle} becomes a standard model for particle filter, such as in \cite{Crisan1998}, Crisan, Gaines and Lyons construct a special sequence of branching particle systems to the solution of the Zakai equation;  and in \cite{Crisan1999}, Crisan and Lyons studies use branching particle systems to study the solution of Kushner–Stratonovitch equation; and in \cite{Whiteley}, some particle stability results are obtained by Whiteley.
However, in its full generality, the system~\eqref{particle} is highly coupled and nonlinear;
in particular, it is not Lipschitz even the coefficients are all bounded and smooth.
As an indispensable intermediate step,
we prove the uniqueness and existence of strong solutions to the particle system~\eqref{particle} by combining the multiplier method, the tightness argument, and the Yamada--Watanabe theorem.

In this paper, we establish the conditional propagation of chaos (CPoC) property of the particle system (\ref{particle}). Specifically, we prove that, as $N$ increases, the weighted empirical measure $\pi_t^N$ converges to the conditional law $\pi_t$ in our target equation (\ref{partial});
moreover, we provide a quantitative characterization for the convergence rate in terms of the Wasserstein distance.
We remark that there are many relevant results in the literature. 
For instance, Zheng~\cite{Zheng} applied CPoC to solve a class of quasilinear equations of parabolic type;
Kurtz and Xiong~\cite{MR1705602,Xiong,Kurtz} constructed a sequence of weighted empirical measures from a particle system to approximate a class of nonlinear SPDE and verify the convergence in the weak topology;
Coghi and Flandolli~\cite{Coghi} established the CPoC property of the particle system with common space-dependent noise;
and Erny, Löcherbach, and Loukianova \cite{Erny} considered jumping particle systems, among others.
To our knowledge, our results differ significantly from existing ones in at least two aspects. 
First, the particle system~\eqref{particle} is not covered by the existing models and its well-posedness result obtained here is new. 
Second, we derive a quantitative CPoC in terms of the Wasserstein distance, whereas a related paper \cite{Xiong} proved the convergence in law without convergence rate and under stronger regularity requirements on the coefficients.
Technically, we develop a multiplier method to deal with the well-posedness and asymptotic behavior of non-Lipschitz particle systems.
In contrast to the localization technique used in~\cite{Buckdahn,Xiong}, our method stands out in its capacity to notably streamline calculations but also to generate new quantitative estimates. 

It is worth noting that the CMVSDE (\ref{partial}) is closely related to the optimal control problem with partial observation (see e.g.~\cite{MR1191160,Pardoux,Tang,Wang,Tan}). 
In particular, in recent years, there has been growing interest in the problem of mean-field control (MFC) and mean-field game (MFG) with partial observation (see e.g.~\cite{Li,Fu,Li2,Sen,Feng,Firoozi}). 
What's more, noticing that $\pi_t$ in (\ref{partial}) satisfies the Zakai equation, our approximation result may be helpful to compute some stochastic control problem driven by Zakai equation (see e.g. \cite{MR1191160,Zhang}) since our work is associated with particle filters, which use particle systems to approximate the solutions in nonlinear filtering problems.  

The rest of this paper is structured as follows: In Section 2, we present the main results including well-posedness and CPoC property of  (\ref{particle}), and error bound of the fully discretization scheme. Auxiliary lemmas are proved in Section 3. Section 4 is devoted to the proof of the main results. The analysis is confirmed by numerical examples in Section 5.

 For convenience, in what follows we shall assume $\sigma_2=0$ in (\ref{partial}). The case of $\sigma_2\neq 0$ is known as the ``correlated noise case" in the nonlinear filtering theory, which is well-understood and without substantial difficulties, although technically slightly more tedious (see e.g.~\cite{Buckdahn}). We prefer not to pursue such complexity in this paper but focus on the conditional McKean-Vlasov nature instead, that is, we set that $\sigma_1=\sigma$ and $\sigma_2=0$.

\section{Main results}

Throughout this paper, we assume that the coefficients $\phi=(b,\sigma,h):\mathbb{R}^n\times \mathcal{P}_1(\mathbb{R}^n)\rightarrow \mathbb{R}^n\times\mathbb{R}^{n\times m}\times\mathbb{R}^k$ is bounded and globally Lipschitz continuous: there is a constant $K\ge 0$ such that
\[ |\phi(x,\mu)|\leq K \quad\text{and}\quad|\phi(x,\mu)-\phi(y,\mu') |\leq K ( |x-y|+W_1 (\mu,\mu')),
\]
where $W_1(\cdot,\cdot)$ denotes the $1$-Wasserstein distance of probability measure:
\[
W_1(\mu_1,\mu_2) := \inf \big\{ \mathbb{E}|\xi_1 - \xi_2|: \mathrm{law}(\xi_1) = 
\mu_1,\,\mathrm{law}(\xi_2) = \mu_2
\big\}.
\]

\subsection{Conditional propagation of chaos}
First, we establish the well-posedness of the particle system (\ref{particle}).
\begin{theorem}\label{Theorem exist-unique}
The particle system (\ref{particle}) admits a unique strong solution $\{(X^{i,N},L^{i,N})\}_{i=1}^N$.
\end{theorem}

Next we use the particle system (\ref{particle}) to approximate the signal $X$ in~\eqref{partial} and the conditional measure $\pi$ in~\eqref{2312261}. To do this, we apply the synchronous coupling technique to generate $\{\bar{X}^i,\bar{L}^i\}_{i=1}^N$, which is the solution to SDE~\eqref{partial} with the driven terms replaced by the Brownian motion $\{ W^i,Y \}_{i=1}^N$ appearing in the particle system (\ref{particle}). 
Since $\{ W^i,Y \}_{i=1}^N$ is a standard Brownian motion, we know that for all $1\leq i\leq N$, the conditional distribution of the synchronous coupling signal $\mathscr{L}^\mathbb{P}(\bar{X}^i_t|\mathcal{F}^Y_t)$ is identical to the conditional distribution of the original signal $\pi$ defined in~\eqref{2312261}. 
For any fixed $N\in\mathbb{N}_+$, we define for all $s\in[0,T]$
\begin{equation}\label{2312252}
\bar{\pi}_s^N=\frac{\sum_{i=1}^N \bar{L}_s^i\delta_{\bar{X}_s^i}}{\sum_{i=1}^N \bar{L}_s^i} = \sum\limits_{i=1}^N \bar{w}_s^{i} \, \delta_{\bar{X}_s^{i}} \quad\text{with}\quad \bar{w}^i_s=\frac{\bar{L}^i_s}{\sum_{i=1}^N \bar{L}^i_s}.
\end{equation}

We have the following approximation theorem.

\begin{theorem}\label{multi}
There is a constant $C=C(K,T,n)$ such that 
\be\nonumber  
\E^{\Q}\Big[\max_{t\in[0,T]}\big|X_t^{i,N}-\bar{X}^{i}_t\big|^2\Psi_t\Big] \leq C N^{-\frac{8}{7n+8}} , \quad\quad \forall i=1,2,...,N,
\ee 
where $n$ is the dimension of synchronous coupling signal $\bar{X}^i$ and the multiplier is defined by
$$
\Psi_t=\exp\bigg\{ -\alpha N\sum\limits_{k=1}^N\int_0^t\left[\frac{1}{N^2}+\Big(\big|w_s^k\big|^2+\big|\bar{w}_s^k\big|^2\Big)\Big(1+\big|X_s^{k,N}\big|^2\Big)\right]\mathrm{d}s \bigg\},
$$
with $\alpha=\alpha(K,T)$ and the random normalized weights ${w}^i_s$ and $\bar{w}^i_s$ given by~\eqref{particle} and~\eqref{2312252}.
\end{theorem}

\begin{remark}
If $h$ is a constant or all the coefficients are independent of $\pi$, then multiplier $\Psi_t$ can be just chosen as $\Psi_t=e^{-\alpha t}$ with $\alpha=\alpha(K,T)>0$. Then we can get that
$$
\E^{\Q}\Big[\max_{t\in[0,T]}\big|X_t^{i,N}-\bar{X}^{i}_t\big|^2\Big] \leq C N^{-\frac{8}{7n+8}} , \quad\quad \forall i=1,2,...,N,
$$
where $C=C(K,T,n)$.
\end{remark}

In view of Theorem \ref{multi}, we can get the mean convergence rate of $\E^{\Q} [|X_t^{i}-\bar{X}^{i}_t|^2 ]$ as follows by estimating multipliers.

\begin{corollary}\label{coro2.1}
For all $i=1,2,...,N$ and $\beta>0$, there exists $C=C(\beta,K,T,n)$ such that
$$
\E^{\Q}\Big[\max\limits_{t\in[0,\tau_\beta\wedge T]}\big|X_t^{i,N}-\bar{X}^{i}_t\big|^2\Big]\leq C N^{-\frac{8}{7n+8}},
$$
where the stopping time is defined as
$$
\tau_\beta:=\inf\bigg\{t:N\sum\limits_{k=1}^N\int_0^t\left[\frac{1}{N^2}+\Big(\big|w_s^k\big|^2+\big|\bar{w}_s^k\big|^2\Big)\Big(1+\big|X_s^{k,N}\big|^2\Big)\right]\mathrm{d}s>\beta\bigg\},
$$
with
$$
\Q[\tau_\beta>T]>1-\frac{C}{\beta}.
$$

Furthermore, it holds that
\be\nonumber
\max\limits_{t\in[0,T]} \E^{\Q}\Big[W_1\big(\pi^N_t,\pi_t\big)^2\Big]+\E^{\Q}\Big[\max\limits_{t\in[0,T]}\big|X_t^{i,N}-\bar{X}^{i}_t\big|^2\Big]\leq C\bigl|\ln(N)\bigr|^{-\frac{1}{2p}},  \quad\quad \forall p>1,
\ee  
where $C=C(K,T,n,p)$, $\pi_t^N$ is the empirical measure defined in~\eqref{particle} and $\pi_t$ is the target measure defined in~\eqref{2312261}.
\end{corollary}

\begin{remark}
The convergence rate $\mathcal{O}(|\ln(N)|^{-\frac{1}{2}+})$ results from the fact that the inverse of multiplier $\Psi^{-1}$ is not bounded, and even not integrable in general, which makes the estimation of convergence rate become intractable. However, we observe that $\ln(\Psi^{-1})$ is integrable, which allows us to get the convergence rate $\mathcal{O}(|\ln(N)|^{-\frac{1}{2}+})$. Although this convergence rate is much slower compared to the classical result, it seems to be the first result in regard to the strong convergence rate for the mean-field particle filters. Moreover, our numerical scheme works very well in some examples even in the cases when coefficients are not bounded (see Section~\ref{secnum}).
\end{remark}

\subsection{Sequential importance sampling: full discretization}
Next, we use the sequential importance sampling to fully discretize the SDE (\ref{partial}), which is a Monte Carlo algorithm based on the CPoC property and the Euler scheme. To approximate the random weight $\{\bar{L}^i\}_{i=1}^{N}$, it is equivalent to consider the random normalized weight $\{\bar{w}^i\}_{i=1}^{N}$ define in~\eqref{2312252}. By Ito's formula, the dynamic of $\bar{w}_t^i$ is governed by
\begin{equation}\label{2312253}
\D \bar{w}_t^{i} = \bar{w}_t^{i}\, H\left( \bar{X}_{t}^{i}, \bar{\pi}_{t}^{N} \right) \D t + \bar{w}_t^{i} \, M^{\text{T}}\left( \bar{X}_{t}^{i}, \bar{\pi}_{t}^{N} \right) \D Y_t,\qquad \bar{w}^{i}_0=1,
\end{equation}
where the coefficient $(M,H):\mathbb{R}^n\times \mathcal{P}_1(\mathbb{R}^n)\rightarrow \mathbb{R}^k\times\mathbb{R}$ is the function satisfying
$$
M(x,\pi)=h(x,\pi)-\int h(x,\pi)\pi(\D x),
$$
and
$$ H(x,\pi)=\Big|\int h(x,\pi)\pi(\D x)\Big|^2-h^{\text{T}} (x,\pi)\int h(x,\pi)\pi(\D x).
$$

Then the full discretization for the synchronous coupling signal $\bar{X}^i$ and for the random normalized weight $\bar{w}^i$ in~\eqref{2312253} is: $\forall 1\leq i\leq N$,
\begin{equation}\label{Euler}
\left\{
\begin{aligned}
&\D X_t^{i,\Delta}=b\big(X_{\theta(t,\Delta)}^{i,\Delta},\pi_{\theta(t,\Delta)}^{N,\Delta}\big)\D t+\sigma\big(X_{\theta(t,\Delta)}^{i,\Delta},\pi_{\theta(t,\Delta)}^{N,\Delta}\big)\D W_t^i,\qquad X_0^{i,\Delta}=x,\\[0.3cm]
&\D w_t^{i,\Delta} = w_t^{i,\Delta} H\big( X_{\theta(t,\Delta)}^{i,\Delta}, \pi_{\theta(t,\Delta)}^{N,\Delta} \big) \D t + w_t^{i,\Delta}M^{\text{T}}\big(X_{\theta(t,\Delta)}^{i,\Delta},\pi_{\theta(t,\Delta)}^{N,\Delta}\big)\D Y_t,\qquad w^{i,\Delta}_0=1,
\end{aligned}
\right.
\end{equation}
where

\begin{itemize} 

\item 
$\Delta\in(0,T)$ is the size of time step, and $\theta: [0,T] \times (0,T)\rightarrow [0,T]$ is a measurable function such that
$$
\theta(t,\Delta)=\Big\lfloor \frac{t}{\Delta}\Big\rfloor \Delta,
$$
with $\lfloor \cdot \rfloor$ being the floor function;

\item
the empirical measure with random weights is defined by
$$
\pi_{t}^{N,\Delta} = \sum\limits_{i=1}^N w_t^{i,\Delta} \,\, \delta_{X_t^{i,\Delta}}\,.
$$

\end{itemize}

We can prove the following approximation results.

\begin{theorem}\label{Euler scheme}
For all $\Delta \in (0,T)$, it holds that
\begin{equation}\nonumber
\E^\Q\Big[\max_{s\in[0,T]}\big|\bar{X}_s^{i}-X_s^{i,\Delta}\big|^2\Big]\leq C\Big\{ \big[\ln\big(\Delta^{-1}\big)\big]^{-\frac{1}{2p}} + \big[\ln\big(N\big)\big]^{-\frac{1}{2p}}\Big\},\quad\quad \forall p>1,
\end{equation}
where $C=C(K,T,n,p)$.
In particular, if we take $\Delta=\frac{T}{N}$, we have
\begin{equation}\nonumber
\E^\Q\Big[\max\limits_{s\in[0,t]}\big|\bar{X}_s^{i}-X_s^{i,\Delta}\big|^2\Big]\leq C\big[\ln\big(N\big)\big]^{-\frac{1}{2p}},\quad\quad \forall p>1,
\end{equation}
where $C=C(K,T,n,p)$.
\end{theorem}

\section{Auxiliary lemmas}
\subsection{Estimates for the weight of empirical measure}
Suppose $\{\tilde{X}_t^i\}_{i=1}^N$ is a solution of the following stochastic equation: for all $1\leq i\leq N$,
\be\label{sde2}
\begin{cases}
\D \tilde{X}_t^i=b(\tilde{X}_t^i,\tilde{\pi}_t)\D t+\sigma(\tilde{X}_t^i,\tilde{\pi}_t)\D W_t^i, \qquad  \tilde{X}_0^i=x,\\  
\D \tilde{L}_t^i=h(\tilde{X}_t^i,\tilde{\pi}_t)\tilde{L}_t^i\D Y_t,\qquad \tilde{L}_0^i=1,
\end{cases}
\ee
where $\tilde{\pi}_t\in \mathcal{P}_1(\mathbb{R}^n)$ is a given random probability measure. For simplicity, we write $(X^{i},L^{i},w^i)$ instead of $(X^{i,N},L^{i,N},w^{i,N})$ in (\ref{particle}) and denote 
\begin{align}\label{tilde_pi}
 \tilde{\pi}^N=\frac{\sum_{i=1}^N  \tilde{L}^i\delta_{\tilde{X}^i}}{\sum_{i=1}^N  \tilde{L}^i}, \quad z^i=\ln(w^i) \quad\text{and} \quad\tilde{z}^i=\ln(\tilde{w}^i).   
\end{align}
Then we have the following estimate of 1-Wasserstein measure.
\begin{lemma}\label{lemma1}
  It holds that
    \begin{align}\nonumber
        \displaystyle W_1\big(\pi^N,\tilde{\pi}^N\big)&\leq \displaystyle \sum\limits_{i=1}^N\big|w^i-\tilde{w}^i\big|\big|X^i\big|+\sum\limits_{i=1}^N \tilde{w}^i\big|X^i-\tilde{X}^i\big| \\
         &\leq \displaystyle \sum\limits_{i=1}^N w^i\vee\tilde{w}^i\big|z^i-\tilde{z}^i\big|\big|X^i\big|+\sum\limits_{i=1}^N \tilde{w}^i\big|X^i-\tilde{X}^i\big|.\nonumber
       \end{align}
\end{lemma}
\begin{proof}
Define $$\underline{\pi}^N=\frac{\sum_{i=1}^N \tilde{L}^i\delta_{X^i}}{\sum_{i=1}^N \tilde{L}^i}.$$ By triangle inequality of Wasserstein space and Kantorovich–Rubinstein formula, we have
\begin{align}
   \displaystyle W_1\big(\pi^N,\tilde{\pi}^N\big)\leq &W_1\big(\pi^N,\underline{\pi}^N\big)+W_1\big(\underline{\pi}^N,\tilde{\pi}^N\big) \nonumber\\
    =&\displaystyle\inf_{\phi\in \text{Lip}(1)}\bigg\{ \bigg|\int \phi(x)\pi^N(\D x)-\int \phi(x)\underline{\pi}^N(\D x)\bigg|\bigg\}\nonumber\\
    &+\inf_{\phi\in \text{Lip}(1)}\bigg\{ \bigg|\int \phi(x)\underline{\pi}^N(\D x)-\int \phi(x)\tilde{\pi}^N(\D x)\bigg|\bigg\}\nonumber\\
    \leq&\displaystyle \sum\limits_{i=1}^N \big|w^i-\tilde{w}^i\big|\big|X^i\big|+\sum\limits_{i=1}^N  \tilde{w}^i\big|X^i-\tilde{X}^i\big|,\nonumber
\end{align}
where the notation $\text{Lip}(1)$ is the space of all Lipschitz functions with Lipschitz constant 1. 

Using the fact $|e^x-e^y|\leq e^{x\vee y}|x-y|$, it yields that 
\be\nonumber
    \big|w^i-\tilde{w}^i\big|\big|X^i\big|\leq w^i\vee\tilde{w}^i\big|z^i-\tilde{z}^i\big|\big|X^i\big|,
\ee
then the proof of Lemma \ref{lemma1} is completed.
\end{proof}

 Recalling the definitions of $M$ and $H$ in (\ref{2312253}), it is easy to see that $M$ and $H$ are bounded functions. Using Ito’s formula, we can get 
\be\label{Ito}
 \D w_t^i= w_t^i H(X_t^i,\pi_t^N)\D t + w_t^i M^{\text{T}}(X_t^i,\pi_t^N)\D Y_t.
\ee
By Ito’s formula (\ref{Ito}), we have
\be\nonumber
\D z_t^i=R(X_t^i,\pi_t^N)\D t+M^{\text{T}}(X_t^i,\pi_t^N)\D Y_t,
\ee
 where $$R(x,\pi)=H(x,\pi)-\frac{1}{2}|M(x,\pi)|^2.$$
Recalling the boundedness and Lipshitz continuity of $h$, we can observe that $R$ and $M$ is bounded and 
\begin{align}\label{LipH}
\left|R(x,\pi)-R(x',\pi')\right|+\left|M(x,\pi)-M(x',\pi')\right|\leq C\left(|x-x'|+W_1\big(\pi,\pi'\big)\right).    
\end{align}
where $C=C(K)$ is a constant.

The following estimate is significant for our proof of uniqueness. 

\begin{lemma}\label{lemma3}
Denote 
$$
\Gamma_t=\exp\bigg(\int_0^t\gamma_s\D s\bigg),
$$ 
then for all $t\in[0,T]$, it holds that
\begin{align*}
\big|X_t^{i}-\tilde{X}_t^{i}\big|^2\Gamma_t
\leq & C\int_0^t W_1\left(\pi_s^N,\tilde{\pi}_s\right)^2\Gamma_s\D s+\int_0^t\big(C+\gamma_s\big)\big|X_s^{i}-\tilde{X}_s^{i}\big|^2\Gamma_s\D s\nonumber\\
&+C\int_0^t\Gamma_s\Big( X_s^{i}-\tilde{X}_s^{i}\Big)^{\text{T}} \Big(\sigma(X_s^i,\pi_s^N)-\sigma(\tilde{X}_s^i,\tilde{\pi_s})\Big)\D W_s^i .
\end{align*}
and
\begin{align}
\big|z_t^{i}-\tilde{z}_t^{i} \big|^2\Gamma_t\leq& C\int_0^t \big| X_s^{i}-\tilde{X}_s^{i}\big|^2 \Gamma_s + W_1\left(\pi_s^N,\tilde{\pi}_s\right)^2\Gamma_s\D s+\int_0^t\big(C+\gamma_s\big)\big|z_s^{i}-\tilde{z}_s^{i}\big|^2\Gamma_s\D s\nonumber\\
    &+C\,\int_0^t\Gamma_s\Big(z_s^{i}-\tilde{z}_s^{i} \Big)\Big(M(X_s^i,\pi_s^N)-M(\tilde{X}_s^i,\tilde{\pi_s}) \Big)^{\text{T}}\D Y_s,\nonumber
\end{align}
where $C=C(K)$ is a constant.
\end{lemma}
\begin{proof}
For simplicity, we denote 
$$
\Delta \tilde{X}_t^i=X_t^i-\tilde{X}_t^i, \quad \Delta \tilde{b}_t^i=b(X_t^i,\pi_t^N)-b(\tilde{X}_t^i,\tilde{\pi_t}), \quad \Delta \tilde{\sigma}_t^i=\sigma(X_t^i,\pi_t^N)-\sigma(\tilde{X}_t^i,\tilde{\pi_t}),
$$
Using Itô’s formula, we obtain
\begin{align}
  \nonumber  \D \big|\Delta \tilde{X}_t^i\big|^2\Gamma_t= &2\Big(\Delta \tilde{X}_t^i\Big)^{\text{T}} \Delta\tilde{\sigma}_t^i\Gamma_t\D W_t^i+2 \Big(\Delta \tilde{X}_t^i\Big)^{\text{T}} \Delta \tilde{b}_t^i\Gamma_t\D t \\
    &\nonumber+ \Tr\left(\left(\Delta\tilde{\sigma}_t^i\right)^{\text{T}}\Delta\tilde{\sigma}_t^i\right)\Gamma_t\D t +\gamma_t\big|\Delta \tilde{X}_t^i\big|^2\Gamma_t\D t.
\end{align}
With the fact that $2ab\leq a^2+b^2$, we have
\begin{align}\label{unique-34}
\big|\Delta \tilde{X}_t^i\big|^2\Gamma_t
\leq & C\int_0^t \big|\Delta \tilde{X}_s^i\big|^2 \Gamma_s + \big|\Delta \tilde{b}_s^i\big|^2\Gamma_s+ \Tr\left(\left(\Delta\tilde{\sigma}_s^i\right)^{\text{T}}\Delta\tilde{\sigma}_s^i\right)\Gamma_s\D s\nonumber\\
&+C\int_0^t\Big(\Delta \tilde{X}_s^i\Big)^{\text{T}} \Delta\tilde{\sigma}_s^i\Phi_s\D W_s^i +\int_0^t\gamma_s\big|\Delta \tilde{X}_s^i\big|^2\Gamma_s\D s.
\end{align}
Combining with (\ref{unique-34}) and the Lipschitz continuity of $(b,\sigma)$, we complete our first estimate in Lemma \ref{lemma3}. The estimate of $\big|z_t^i- \tilde{z}_t^i\big|^2\Gamma_t$ is similar because of the Lipschtiz continuity of $(R,H)$.
\end{proof}

Next lemma shows that $\pi_t^N$ is continuous in probability.
\begin{lemma}\label{lemma4}
Suppose that there exists $C=C(K,T)$ satisfying for all $1\leq i\leq N$
\be
\E^{\Q}\Big[\max\limits_{t\in[0,T]}\big|\tilde{X}_t^i\big|^4+N^4\big|\tilde{w}_t^i\big|^4\Big]\leq C\quad\text{and}\quad \E^{\Q}\Big[\big|\tilde{X}_t^i-\tilde{X}_s^i\big|^4+N^4\big|\tilde{w}_t^i-\tilde{w}_s^i\big|^4\Big]\leq C|t-s|^2.\nonumber
\ee
Then for $\tilde{\pi}^N$ defined in (\ref{tilde_pi}), it holds that
\be\nonumber \E^{\mathbb{Q}}\Big[W_1\big(\tilde{\pi}_t^N,\tilde{\pi}_s^N\big)^2\Big]\leq C\,|t-s|, \quad\forall \, s,t\in[0, T],
\ee 
where $C=C(K,T)$.
\end{lemma}

\begin{proof}
According to Lemma \ref{lemma1}, 
\be\nonumber
W_1\big(\tilde{\pi}_t^N,\tilde{\pi}_s^N\big)\leq \sum\limits_{i=1}^N \tilde{w}_s^i\big|\tilde{X}_t^i-\tilde{X}_s^i\big|+\sum\limits_{i=1}^N\big|\tilde{X}_t^i\big|\big|\tilde{w}_t^i-\tilde{w}_s^i\big|.
\ee
Consequently, it follows that  
\begin{align}\label{lemma4.1}
\E^{\Q}\Big[W_1\big(\tilde{\pi}_t^N,\tilde{\pi}_s^N\big)^2\Big]\leq & C\E^{\Q}\bigg[\bigg(\sum\limits_{i=1}^N \tilde{w}_s^i\big|\tilde{X}_t^i-\tilde{X}_s^i\big|\bigg)^2\bigg]+C\E^{\Q}\bigg[\bigg(\sum\limits_{i=1}^N\big|\tilde{X}_t^i\big|\big|\tilde{w}_t^i-\tilde{w}_s^i\big|\bigg)^2\bigg]\nonumber\\
\leq& C\E^{\Q}\bigg[\sum\limits_{i=1}^N\big|\tilde{w}_s^i\big|^2\sum\limits_{j=1}^N\big|\tilde{X}_t^j-\tilde{X}_s^j\big|^2\bigg]+C\E^{\Q}\bigg[\sum\limits_{i=1}^N\big|\tilde{X}_t^i\big|^2\sum\limits_{j=1}^N\big|\tilde{w}_t^j-\tilde{w}_s^j\big|^2\bigg]\nonumber\\
\leq& C\sum\limits_{1\leq i,j\leq N}\E^{\Q}\Big[\big|\tilde{w}_s^i\big|^4\Big]^{\frac{1}{2}}\E^{\Q}\Big[\big|\tilde{X}_t^j-\tilde{X}_s^j\big|^4\Big]^{\frac{1}{2}}
\nonumber\\
& +C\sum\limits_{1\leq i,j\leq N}\E^{\Q}\Big[\big|\tilde{X}_t^i\big|^4\Big]^{\frac{1}{2}}\E^{\Q}\Big[\big|\tilde{w}_t^j-\tilde{w}_s^j\big|^4\Big]^{\frac{1}{2}}.
\end{align}
By our conditions, we can conclude the proof of the lemma.
\end{proof}

\subsection{Wasserstein convergence of weighted empirical measures}
The convergence of
unweighted empirical measures in the Wasserstein distance has been extensively investigated in the literature. Here we follow the idea of J. Horowitz and R. Karandikar (\cite{Horowitz}), which helps us to obtain the following mean convergence rate of the Wasserstein space first.
\begin{lemma}\label{convergence rate}
It holds that
\be\nonumber
\max\limits_{t\in[0,T]} \E^{\Q}\left[ W_1\left(\bar{\pi}_t^{N},\pi_t\right)^4\right]\leq CN^{-\frac{16}{7n+8}},
\ee 
where $C=C(K,T,n)$, $n$ is the dimension of the signal process, $\pi_t^N$ is the empirical measure defined in~\eqref{particle} and $\pi_t$ is the original measure defined in~\eqref{2312261}.
\end{lemma}
\begin{proof}
   
We write $\Phi_{d}\sim N(0,d^2I)$ to indicate $\Phi_d$ is the normal distribution on $\mathbb{R}^n$ with mean $0$ and variance $d^2I$. $\phi_{d}$ is the probability density function corresponding to $\Phi_{d}$. For any probability measure $\mu$ on $\mathbb{R}$, let $\mu^{d}=\Phi_d*\mu$ be the convolution of $\Phi_d$ and $\mu$.
Consequently, we can observe that  
\be\label{all_5}
W_1\left(\mu,\mu_d\right) \leq Cd.
\ee
To see this, let $\xi_1$ and $\xi_2$ be the independent random variables with laws $\mu$ and $\phi_d$ respectively, then $(\xi_1,\xi_1+\xi_2)$ is a coupling of $\mu$ and $\mu_d$, and 
$$ 
W_1\left(\mu,\mu_d\right)\leq \int_{\mathbb{R}}|x|\phi_d\D x=Cd.
$$ 

Hence, by triangle inequality of Wasserstein space, we have  
\be\label{all_6}
W_1\big(\bar{\pi}_t^N,\pi_t\big)^4\leq 27  W_1\big(\bar{\pi}_t^{N,d},\pi^d_t\big)^2+27  W_1\big(\pi_t,\pi^d_t\big)^2+27 W_1\big(\bar{\pi}_t^N,\bar{\pi}_t^{N,d}\big)^2.
\ee

Let $f_t^{N,d}$ and $f_t^d$ be the probability density functions of $\bar{\pi}_t^{N,d}$ and $\pi^d_t$. By Kantorovich-Rubinstein theorem and Holder's inequality, it follows that
\begin{align*}\nonumber
W_1\big(\bar{\pi}_t^{N,d},\pi^d_t\big)&=\inf_{\psi\in \text{Lip}(1)}\bigg\{\int\psi(x)f_t^{N,d}(x)\D x-\int\psi(x)f_t^{d}(x)\D x\bigg\}\\
&\leq \int |x|\big|f_t^{N,d}(x)-f_t^d(x)\big|\D x\nonumber\\
&\leq C \Big(\int \big(|x|^{3n+7}+1\big)\big|f_t^{N,d}(x)-f_t^d(x)\big|^4\D x\Big)^{\frac{1}{4}}.\nonumber
\end{align*}
Therefore, we have
\be\label{all_1}
\E^{\Q}\Big[ W_1\big(\bar{\pi}_t^{N,d},\pi^d_t\big)^4\Big]\leq C\int \big(|x|^{3n+7}+1\big)\E^{\Q}\Big[\big|f_t^{N,d}(x)-f_t^d(x)\big|^4\Big]\D x.
\ee
By definition, we can actually get the expression of $f_t^{N,d}(x)$ and $f_t^d(x)$:
\be\nonumber
f_t^{N,d}(x)=\frac{\sum_{i=1}^N\bar{L}_t^i\phi_d\big(x-\bar{X}_t^i\big)}{\sum_{i=1}^N\bar{L}_t^i} \quad\quad\text{and}\quad\quad f_t^d(x)=\frac{\E^{\Q}\left[\bar{L}_t^1\phi_d\left(x-\bar{X}_t^1\right)\big|\mathcal{F}_t^Y\right]}{\E^{\Q}\left[\bar{L}_t^1\big|\mathcal{F}_t^Y\right]},
\ee
where we use the fact that for all $i\geq 1$,
$$\E^{\Q}\Big[\bar{L}_t^i\phi_d\big(x-\bar{X}_t^i\big)\Big|\mathcal{F}_t^Y\Big]=\E^{\Q}\Big[\bar{L}_t^1\phi_d\big(x-\bar{X}_t^1\big)\Big|\mathcal{F}_t^Y\Big]\quad\text{and}\quad\E^{\Q}\Big[\bar{L}_t^i\Big|\mathcal{F}_t^Y\Big]=\E^{\Q}\big[\bar{L}_t^1|\mathcal{F}_t^Y\big].$$
Consequently, it follows that
\begin{align}\label{all_2}
\big|f_t^{N,d}(x)-f_t^d(x)\big|^4\leq &
8\,\bigg|\frac{\frac{1}{N}\sum_{i=1}^N\bar{L}_t^i\phi_d\left(x-\bar{X}_t^i\right)-\E^{\Q}\left[\bar{L}_t^1\phi_d\left(x-\bar{X}_t^1\right)|\mathcal{F}_t^Y\right]}{\frac{1}{N}\sum_{i=1}^N\bar{L}_t^i}\bigg|^4\nonumber\\ &+8\bigg|\E^{\Q}\Big[\bar{L}_t^1\phi_d\left(x-\bar{X}_t^1\right)\Big|\mathcal{F}_t^Y\Big]\bigg|^4\bigg|\Big(\frac{1}{N}{\sum_{i=1}^N\bar{L}_t^i}\Big)^{-1}-\Big(\E^{\Q}\big[\bar{L}_t^1\big|\mathcal{F}_t^Y\big]\Big)^{-1}\bigg|^4\nonumber\\
=&A_1+A_2.
\end{align}

Firstly, the term $A_1$ satisfies 
\be\label{A1_1}
\E^{\Q}\big[A_1\big]\leq C\,\E^{\Q}\bigg[\Big|\frac{1}{N}\sum^N_{i=1}\bar{L}_t^i\phi_d(x-\bar{X}_t^i)-\E^{\Q}\Big[\bar{L}_t^1\phi_d\big(x-\bar{X}_t^1\big)\Big|\mathcal{F}_t^Y\Big]\Big|^8\bigg]^{\frac{1}{2}} \E^{\Q}\bigg[\Big|\frac{1}{N}\sum_{i=1}^N\bar{L}_t^i\Big|^{-8}\bigg]^{\frac{1}{2}} .
\ee
Notice for all $1\leq i\leq N$, 
$$
\E^{\Q}\bigg[\max\limits_{t\in[0,T]}\big|\bar{L}_t^i\big|^{-8}\bigg]\leq C.
$$ 
Then by Jensen's inequality, it holds that
\be\label{A1_2}
\E^{\Q}\bigg[\Big|\frac{1}{N}\sum\limits_{i=1}^N\bar{L}_t^i\Big|^{-8}\bigg]\leq \E^{\Q}\bigg[\frac{1}{N}\sum\limits_{i=1}^N\big|\bar{L}_t^i\big|^{-8}\bigg]\leq C.
\ee
Note that $\{(\bar{X}_t^i,\bar{L}_t^i)\}_{i\geq 1}$ are the i.i.d. random variables under $\mathbb{P}\big[\cdot\big|\mathcal{F}_t^Y\big]$. We denote 
$$
g_d\left(x,\bar{X}_t^i,\bar{L}_t^i\right)=\bar{L}_t^i\phi_d\big(x-\bar{X}_t^i\big)-\E^{\Q}\Big[\bar{L}_t^i\phi_d\big(x-\bar{X}_t^i\big)\Big|\mathcal{F}_t^Y\Big],
$$ 
then we can deduce that $\{g_d (x,\bar{X}_t^i,\bar{L}_t^i)\}_{i=1}^N$ are the i.i.d. r.v.'s under $\mathbb{P}\big[\cdot\big|\mathcal{F}_t^Y\big]$. Hence, it follows that
\begin{align}\label{sum1}
&\E^{\Q}\bigg[\Big|\frac{1}{N}\sum^N_{i=1}\bar{L}_t^i\phi_d(x-\bar{X}_t^i)-\E^{\Q}\Big[\bar{L}_t^1\phi_d\big(x-\bar{X}_t^1\big)\Big|\mathcal{F}_t^Y\Big]\Big|^8\bigg|\mathcal{F}_t^Y\bigg]\nonumber\\
\leq &\frac{C}{N^8}\sum\limits_{1\leq i_1,i_2,i_3,i_4\leq N}\E^{\Q}\bigg[\prod\limits_{m=1}^4\Big|g_d\big(x,\bar{X}_t^{i_m},\bar{L}_t^{i_m}\big)\Big|^2\bigg|\mathcal{F}_t^Y\bigg].
\end{align}
The right hand side in (\ref{sum1}) has the upper bound estimates:
\begin{align} \nonumber
\E^{\Q}\bigg[\prod\limits_{m=1}^4\Big|g_d\big(x,\bar{X}_t^{i_m},\bar{L}_t^{i_m}\big)\Big|^2\bigg|\mathcal{F}_t^Y\bigg]
\leq \prod\limits_{m=1}^4\E^{\Q}\Big[\big|g_d\big(x,\bar{X}_t^{i_m},\bar{L}_t^{i_m}\big)\big|^8\Big|\mathcal{F}_t^Y\Big]^{\frac{1}{4}}=\E^{\Q}\left[\left|g_d\left(x,\bar{X}_t^1,\bar{L}_t^1\right)\right|^8|\mathcal{F}_t^Y\right].\nonumber
\end{align}
According to (\ref{sum1}), we have
\begin{align}\label{A1_4}
\displaystyle\E^{\Q}\bigg[\Big|\frac{1}{N}\sum^N_{i=1}\bar{L}_t^i\phi_d(x-\bar{X}_t^i)-\E^{\Q}\Big[\bar{L}_t^1\phi_d\big(x-\bar{X}_t^1\big)\Big|\mathcal{F}_t^Y\Big]\Big|^8\bigg|\mathcal{F}_t^Y\bigg] \leq \frac{C}{N^4}\E^{\Q}\Big[\big|g_d\left(x,\bar{X}_t^1,\bar{L}_t^1\right)\big|^8\Big|\mathcal{F}_t^Y\Big].
\end{align}
Combining (\ref{A1_1}) and (\ref{A1_4}), we can get the upper bound of $\E^{\Q}[A_1]$ with the help of Jensen's inequality
\begin{align}\label{A1_3}
\E^{\Q}\big[A_1\big]&\leq \frac{C}{N^2}\E^{\Q}\bigg[\Big|\E^{\Q}\big[\big|g_d\left(x,\bar{X}_t^1,\bar{L}_t^1\right)\big|^4|\mathcal{F}_t^Y\big]\Big|^\frac{1}{2}\bigg]\leq \frac{C}{N^2}\E^{\Q}\bigg[\big|g_d\big(x,\bar{X}_t^1,\bar{L}_t^1\big)\Big|^8\bigg]^{\frac{1}{2}}\nonumber\\
&\leq \frac{C}{N^2}\E^{\Q}\bigg[\Big|\bar{L}_t^1\phi_d\big(x-\bar{X}_t^1\big)\Big|^8\bigg]^{\frac{1}{2}}.
\end{align}

Next, we estimate $\E^{\Q}[A_2]$ by Cauchy's inequality:
\begin{align}\label{A2_1}
\E^{\Q}\big[A_2\big]& \leq \E^{\Q}\bigg[\bigg|\frac{1}{\frac{1}{N}\sum_{i=1}^N\bar{L}_t^i}-\frac{1}{\E^{\Q}\big[\bar{L}_t^1\big|\mathcal{F}_t^Y\big]}\bigg|^8\bigg]^{\frac{1}{2}}\E^{\Q}\bigg[\Big|\E^{\Q}\big[\bar{L}_t^1\phi_d\big(x-\bar{X}_t^1\big)\big|\mathcal{F}_t^Y\big]\Big|^8\bigg]^{\frac{1}{2}}\nonumber\\
&\leq \E^{\Q}\bigg[\Big|\E^{\Q}\big[\bar{L}_t^1\big|\mathcal{F}_t^Y\big]-\frac{1}{N}\sum\limits_{i=1}^N\bar{L}_t^i\Big|^{16}\bigg]^{\frac{1}{4}}\E^{\Q}\bigg[\Big|\E^{\Q}\big[\bar{L}_t^1|\mathcal{F}_t^Y\big]\Big|^{-16}\Big|\frac{1}{N}\sum\limits_{i=1}^N\bar{L}_t^i\Big|^{-16}\bigg]^{\frac{1}{4}}\nonumber\\
&\quad\times \E^{\Q}\Big[\big|\bar{L}_t^1\phi_d\big(x-\bar{X}_t^1\big)\big|^8\Big]^{\frac{1}{2}}.
\end{align}
By the similar calculation that deduces (\ref{A1_2}) and (\ref{A1_4}), we have
\begin{align}\label{A2_2}   
\E^{\Q}\bigg[\Big|\E^{\Q}\big[\bar{L}_t^1|\mathcal{F}_t^Y\big]\Big|^{-16}\Big|\frac{1}{N}\sum\limits_{i=1}^N\bar{L}_t^i\Big|^{-16}\bigg]\leq C,
\end{align}
and
\be\label{A2_3}
\E^{\Q}\bigg[\Big|\E^{\Q}\big[\bar{L}_t^1\big|\mathcal{F}_t^Y\big]-\frac{1}{N}\sum\limits_{i=1}^N\bar{L}_t^i\Big|^{16}\bigg]\leq \frac{C}{N^8}.
\ee
In view of (\ref{A2_1}), (\ref{A2_2}) and (\ref{A2_3}), we obtain 
\be\label{A2_4}
\E^{\Q}\big[A_2\big]\leq \frac{C}{N^2}\E^{\Q}\Big[\big|\bar{L}_t^1\phi_d\big(x-\bar{X}_t^1\big)\big|^8\Big]^{\frac{1}{2}}.
\ee
By (\ref{all_2}), (\ref{A1_3}) and (\ref{A2_4}), it follows that
\be\nonumber
\E^{\Q}\Big[ \big|f_t^{N,d}(x)-f_t^d(x)\big|^4\Big]\leq \frac{C}{N^2}\E^{\Q}\Big[\big|\bar{L}_t^1\phi_d\big(x-\bar{X}_t^1\big)\big|^8\Big]^{\frac{1}{2}}.
\ee
According to (\ref{all_1}) and Cauchy's inequality, we have
\begin{align}\label{all_wa}
\E^{\Q}\Big[ W_1\big(\bar{\pi}_t^{N,d},\pi^d_t\big)^4\Big]&\leq \frac{C}{N^2}\int \big(|x|^{3n+7}+1\big)\E^{\Q}\Big[\big|\bar{L}_t^1\phi_d\big(x-\bar{X}_t^1\big)\big|^8\Big]^{\frac{1}{2}}\D x\nonumber\\
&\leq \frac{C}{N^2}\left(\int \left(|x|^{6n+16}+1\right)\E^{\Q}\left[\left|\bar{L}_t^1\phi_d\left(x-\bar{X}_t^1\right)\right|^8\right]\D x\right)^{\frac{1}{2}}.
\end{align}
By simple calculation, we can check that
\be\nonumber
\phi_d(x)^8=Cd^{-7n}\phi_{\frac{d}{2\sqrt{2}}}(x).
\ee
Hence, it holds that
\begin{align}\nonumber
&\int \big(|x|^{6n+16}+1\big)\E^{\Q}\left[\left|\bar{L}_t^1\phi_d\left(x-\bar{X}_t^1\right)\right|^8\right]\D x\\
=&\, Cd^{-7n}\E^{\Q}\Big[\big|\bar{L}_t^1\big|^8\int \big(|x|^{6n+16}+1\big)\phi_{\frac{d}{2\sqrt{2}}}\big(x-\bar{X}_t^1\big)\D x\Big]\nonumber\\
\leq &\,Cd^{-7n}\E^{\Q}\Big[\big|\bar{L}_t^1\big|^8\int \big(|x|^{6n+16}+\big|\bar{X}_t^1\big|^{6n+16}+1\big)\phi_{\frac{d}{2\sqrt{2}}}(x)\D x\Big]\leq Cd^{-7n}.\nonumber
\end{align}
In view of (\ref{all_wa}), we have
\begin{align}\label{all_4}
\E^{\Q}\Big[ W_1\big(\bar{\pi}_t^{N,d},\pi^d_t\big)^4\Big]\leq \frac{C}{N^2d^{\frac{7}{2}n}}.
\end{align}
Choosing $d=N^{-\frac{4}{7n+8}}$ and combining (\ref{all_5}), (\ref{all_6}) and (\ref{all_4}), we obtain 
\be\nonumber
\E^{\Q}\Big[ W_1\big(\bar{\pi}_t^{N},\pi_t\big)^4\Big]\leq CN^{-\frac{16}{7n+8}}.
\ee

At last, we notice constant $C=C(K,T,n)$, so we can conclude that
\be\nonumber
\max\limits_{t\in[0,T]} \E^{\Q}\Big[ W_1\big(\bar{\pi}_t^{N},\pi_t\big)^4\Big]\leq CN^{-\frac{16}{7n+8}}.
\ee
The proof is complete.
\end{proof}

\section{Proof of main results}
\subsection{Proof of Theorem \ref{Theorem exist-unique}}
We first give the proof of uniqueness of solution to SDE (\ref{particle}).

\begin{theorem}\label{theorem1}
    Given probability space $(\Omega, \mathcal{F}, \mathbb{Q},\{\mathcal{F}_t\},\{W_t^i\}_{i=1}^N,Y )$, suppose that $\{(X_t^i,L_t^i)\}_{i=1}^N$ and $\{ (\hat{X}_t^i,\hat{L}_t^i)\}_{i=1}^N$ are two strong solutions to the SDE~\eqref{particle}. Then, it holds that for all $i=1,2,...,N$,
$$
(X_t^i,L_t^i)=(\hat{X}_t^i,\hat{L}_t^i),\quad \quad \mathbb{Q}-a.e.
$$
\end{theorem}
\begin{proof}
Denote 
$$
\hat{\pi}_t^N=\frac{\sum_{i=1}^N\hat{L}_t^i\delta_{\hat{X}_t^i}}{\sum_{i=1}^N\hat{L}_t^i}
$$
and define the function 
$$
\Phi_t=\exp\bigg(-\alpha N\sum\limits_{k=1}^N\int_0^t \Big[\frac{1}{N^2}+\Big(\big|w_s^k\Big|^2+\Big|\hat{w}_s^k\big|^2\Big)\Big(1+\big|X_s^k\big|^2\Big)\Big]\mathrm{d}s\bigg),
$$ where $\alpha>0$ is an undetermined constant.
Denote 
$$
\Delta \hat{X}_t^i=X_t^i-\hat{X}_t^i, \quad\Delta \hat{z}_t^i=z_t^i-\hat{z}_t^i ,  \quad\Delta \hat{\sigma}_t^i=\sigma(X_t^i,\pi_t^N)-\sigma(\hat{X}_t^i,\hat{\pi}_t^N), \quad\Delta \hat{M}_t^i=M(X_t^i,\pi_t^N)-M(\hat{X}_t^i,\hat{\pi}_t^N).
$$ 
Utilizing Lemma\ref{lemma3}, we have
\begin{align}\label{u-4}
    \big|\Delta \hat{X}_t^{i}\big|^2\Phi_t
\leq  C\int_0^t W_1\left(\pi_s^N,\hat{\pi}^N_s\right)^2\Phi_s\D s+C\int_0^t\big|\Delta \hat{X}_s^{i}\big|^2\Phi_s\D s+C\int_0^t\Phi_s\Big( \Delta \hat{X}_t^{i}\Big)^{\text{T}} \Delta\hat{\sigma}_s^i\D W_s^i .
\end{align}
According to Lemma \ref{lemma1}, it holds that 
\be\label{u-5}
W_1\Big(\pi_t^N,\hat{\pi}_t^N\Big)^2\Phi_t\leq 2\bigg(\sum\limits_{i=1}^N\big| X_t^i\big|\big|\Delta \hat{z}_t^i\big|\big|w_t^i\vee\hat{w}_t^i\big|\bigg)^2\Phi_t+2\bigg(\sum\limits_{i=1}^N \hat{w}_t^i\big|\Delta \hat{X}_t^i\big|\bigg)^2\Phi_t.
\ee
By Cauchy's inequality, we have
\be\label{u-6}
\bigg(\sum\limits_{i=1}^N\big| X_t^i\big|\big|\Delta \hat{z}_t^i\big|\big|w_t^i\vee\hat{w}_t^i\big|\bigg)^2\Phi_t
\leq \sum\limits_{j=1}^N\big|\Delta\hat{z}_t^j\big|^2\sum\limits_{k=1}^N\big|X_t^k\big|^2\big|w_t^k\vee\hat{w}_t^k\big|^2\Phi_t,
\ee
and
\be\label{unique-16}
\bigg(\sum\limits_{i=1}^N \hat{w}_t^i\big|\Delta \hat{X}_t^i\big|\bigg)^2\Phi_t\leq \sum\limits_{j=1}^N \big|\hat{w}_t^j\big|^2\sum\limits_{k=1}^N \big|\Delta \hat{X}_t^k\big|^2\Phi_t.
\ee  
Plugging (\ref{u-5}), (\ref{u-6}) and (\ref{unique-16}) into (\ref{u-4}) and summing over $i$, we obtain that
\begin{align}\label{u-8}
\frac{1}{N}\sum\limits_{i=1}^N\left|\Delta \hat{X}_t^i\right|^2\Phi_t \leq & \, C\int_0^t \bigg(1+N\sum\limits_{k=1}^N\big|\hat{w}_s^k\big|^2\bigg)\frac{1}{N}\sum\limits_{i=1}^N\big|\Delta \hat{X}_s^i\big|^2 \Phi_s \D s+\frac{C}{N}\sum_{i=1}^N\int_0^t\Big(\Delta \hat{X}_s^i\Big)^{\text{T}} \Delta\hat{\sigma}_s^i\Phi_s\D W_s^i\nonumber\\
& + C\int_0^t\sum\limits_{k=1}^N \Big(\big|w_s^k\big|^2+\big|\hat{w}_s^k\big|^2\Big)\big|X_s^k\big|^2 \sum\limits_{i=1}^N\big|\Delta \hat{z}_s^i\big|^2\Phi_s\D s\nonumber\\
& -\alpha \int_0^t \sum\limits_{k=1}^N \bigg[\frac{1}{N^2}+\Big(\big|w_s^k\big|^2+\big|\hat{w}_s^k\big|^2\Big)\Big(1+\big|X_s^k\big|^2\Big)\bigg]\sum\limits_{i=1}^N\big|\Delta \hat{X}_s^i\big|^2\Phi_s\D s.
\end{align}

Similarly, it yields from Lemma \ref{lemma3} that
\begin{align}\label{unique-17}
\frac{1}{N}\sum\limits_{i=1}^N\big| \Delta\hat{z}_t^i\big|^2\Phi_t\leq& \,C \int_0^t  \Big(1+N\sum\limits_{i=1}^N\big|\hat{w}_s^i\big|^2\Big)\frac{1}{N}\sum\limits_{i=1}^N\big| \Delta \hat{X}_s^i\big|^2\Phi_s \D s+\frac{C}{N}\sum_{i=1}^N\int_0^t\Delta \hat{z}_s^i\Gamma_s\Big(\Delta \hat{M}_s^i\Big)^{\text{T}}\D Y_s\nonumber\\
    &+\big(C-\alpha \big) \int_0^t \sum\limits_{k=1}^N \bigg[\frac{1}{N^2}+\Big(\big|w_s^k\big|^2+\big|\hat{w}_s^k\big|^2\Big)\Big(1+\big|X_s^k\big|^2\Big)\bigg]\sum\limits_{i=1}^N\big|\Delta \hat{z}_s^i\big|^2\Phi_s\D s.
\end{align}
Combining (\ref{u-8}) and (\ref{unique-17}) and taking $\alpha$ large enough, we can get
\be\label{u-11}
\sum\limits_{i=1}^N\E^{\mathbb{Q}}\Big[\big|\Delta \hat{X}_t^i\big|^2\Phi_t+\big|\Delta \hat{z}_t^i\big|^2\Phi_t\Big]\leq 0.
\ee
Since $\Phi_t>0$, we can get $\Delta \hat{X}_t^i=0$ and $\Delta \hat{z}_t^i=0$ a.s.. Furthermore, we can conclude with $\Delta \hat{L}_t^i=0$ a.s..
\end{proof}

Let us recall the classic result of the equivalence of the weak solution of a functional SDE and the solution to the corresponding local martingale problem. More details can be found in \cite{Ikeda}.
    
\begin{theorem}\label{theorem2}
The existence of the weak solution $(\Omega,\mathcal{F},\mathbb{Q},\{X_t^i\}_{i=1}^N,\{L_t^i\}_{i=1}^N,\{W_t^i\}_{i=1}^N,Y )$ to the SDE~\eqref{particle} is equivalent to the existence of a probability space $(\tilde{\Omega},\tilde{\mathcal{F}},\tilde{\mathbb{Q}})$ satisfying that for all $f\in C_b^2 (\mathbb{R}^{n\times N}\times \mathbb{R}^N )$,
\be\label{martingale}
f(X_t,L_t)-f(X_0,L_0)-\int_0^t(Af)(s,X,L)\D s \in \mathcal{M}^{c,loc}(\tilde{\mathbb{Q}}),
\ee
where 

\begin{itemize}

\item
$X=(X^1,\cdots,X^N)$ and $L=(L^1,\cdots,L^N)$;

\item
$\mathcal{M}^{c,loc}(\tilde{\mathbb{Q}})$ is the family of all the continuous local martingales under the probability $\tilde{\mathbb{Q}}$;

\item
the operator $A$ is defined by
\begin{align*}
(Af)(t,X,L)=&\sum\limits_{i=1}^N b^{\text{T}}(X_t^i,\pi_t^N)\frac{\partial f}{\partial x_i}(X_t)+\frac{1}{2}\sum\limits_{i=1}^N\Tr\Big(\frac{\partial^2f}{\partial x_i^2}(X_t,L_t)\sigma\sigma^{\text{T}}(X_t^i,\pi_t^N)\Big)\nonumber\\
&+\frac{1}{2} \sum\limits_{i=1}^N \big|h(X_t^i,\pi_t^N)\big|^2 \big|L_t^i\big|^2 \frac{\partial^2f}{\partial l_i^2}(X_t,L_t)\nonumber
\end{align*}
with
$$
\frac{\partial f}{\partial x_i}=\left(\frac{\partial f}{\partial x_{i,1}},\cdots,\frac{\partial f}{\partial x_{i,n}}\right)^{\text{T}}\in \mathbb{R}^n, \quad
\frac{\partial^2 f}{\partial x^2_i}
=\left[\begin{array}{ccc}
  \frac{\partial^2 f}{\partial x^2_{i,1}}   & \cdots&  \frac{\partial^2 f}{\partial x_{i,1}\partial x_{i,n}}\\
   \vdots  & \ddots &  \vdots \\
   \frac{\partial^2 f}{\partial x_{i,n}\partial x_{i,1}}   & \cdots&  \frac{\partial^2 f}{\partial x^2_{i,n}}
\end{array}
\right]\in \mathbb{R}^{n\times n}
$$
and
$$
\pi_t^N= \frac{\sum_{i=1}^NL_t^i\delta_{X_t^i}}{\sum_{i=1}^N L_t^i}.
$$

\end{itemize}

\end{theorem}

Now we show the existence of the strong solution.
\begin{theorem}\label{theorem4.3}
There exists a weak solution $(\Omega,\mathcal{F},\mathbb{Q},\{X_t^i\}_{i=1}^N,\{L_t^i\}_{i=1}^N,\{W_t^i\}_{i=1}^N,Y )$ to the SDE~\eqref{particle}.
\end{theorem}
\begin{proof}
According to Theorem \ref{theorem2}, it remains to check the local martingale problem (\ref{martingale}).   

First we can construct a series of stochastic process $X^{(k)}=\{X^{i,k}\}_{i=1}^N$, $L^{(k)}=\{L^{i,k}\}_{i=1}^N$ and $X^{(k)}=\{X^{i,k}\}_{i=1}^N$, $z^{(k)}=\{z^{i,k}\}_{i=1}^N$ with $X_0^{i,k}=x$ and $z_0^{i,k}=ln(\frac{1}{N})$ on the canonical probability space $(\tilde{\Omega},\tilde{\mathcal{F}},\tilde{\mathbb{Q}},\{W_t^i\}_{i=1}^N,Y)$ satisfying $\forall\, t\in[{jT/2^k},{(j+1)T/2^k}]$,
\begin{equation}\nonumber
\left\{
\begin{aligned}                       
&X_t^{i,k}=X_{jT/2^k}^{i,k}+b\big(X_{jT/2^k}^{i,k},\pi^{N,k}_{jT/2^k}\big)\Big(t-\frac{j}{2^k}T\Big)+\sigma\big(X_{jT/2^k}^{i,k},\pi^{N,k}_{jT/2^k}\big)\Big(W^i_t-W^i_{jT/2^k}\Big),\\
&z^{i,k}_t=z_{jT/2^k}^{i,k}+R \big( X_{jT/2^k}^{i,k},  \pi_{jT/2^k}^{N,k}\big)\Big( t-\frac{j}{2^k}T \Big)+H \big( X_{jT/2^k}^{i,k},  \pi_{jT/2^k}^{N,k} \big) \left(Y_t-Y_{jT/2^k}\right),\\
  &L_t^{i,k}=\exp\bigg(\int_0^t h^{\text{T}}\big(X_s^{i,k},\pi_s^{N,k}\big)\D Y_s-\frac{1}{2}\int_0^t \Big|h\big(X_s^{i,k},\pi_s^{N,k}\big)\Big|^2\D s\bigg), \\
\end{aligned}
\right.
\end{equation}
where $\{W_t^i\}_{i=1}^N$, $Y$ are independent Brownian motions in $(\tilde{\Omega},\tilde{\mathcal{F}},\tilde{\mathbb{Q}})$, 
$$w_t^{i,k}=\exp\big(z_t^{i,k}\big)\quad\text{and}\quad \pi_t^{N,k}=\sum\limits_{i=1}^Nw_t^{i,k}\delta_{X_t^{i,k}}.$$
In view of the boundedness of $b$, $\sigma$, $h$, it follows
\be\nonumber
\sup\limits_{k}\sup\limits_{t\in[0,T]}\E^{\tilde{\Q}}\Big[\big|X^{i,k}_t\big|^4+\big|L^{i,k}_t\big|^4+\big|z^{i,k}_t\big|^4\Big]\leq \,C,
\ee
and
\be\nonumber
\sup\limits_{k}\E^{\tilde{\Q}}\Big[\big|X^{i,k}_t-X^{i,k}_s\big|^4+\big|L^{i,k}_t-L^{i,k}_s\big|^4+\big|z^{i,k}_t-z^{i,k}_s\big|^4\Big]&\leq C|t-s|^2,\ \ \forall\, 0\leq s\leq t\leq T.
\ee

Due to the Theorem 4.2 in \cite{Ikeda}, there exist $(n+1)N+1$-dimensional continuous processes 
$$
\left(\hat{X},\hat{L},\hat{z}\right)\quad \text{and} \quad\left\{\left(\hat{X}^{(k_n)},\hat{L}^{(k_n)},\hat{z}^{(k_n)}
\right)\right\}_{n\geq 1}
$$ 
defined on a completed probability space $(\hat{\Omega},\hat{\mathcal{F}},\hat{\mathbb{Q}})$, such that $(\hat{X}^{(k_n)},\hat{L}^{(k_n)},\hat{z}^{(k_n)})$ has the same distribution of $(X^{(k_n)},L^{(k_n)},z^{(k_n)})$, and $\{ ( \hat{X}^{(k_n)}_t, \hat{L}^{(k_n)}, \hat{z}^{(k_n)}) \}_{n\geq 1}$  converge to $(\hat{X}_t,\hat{L}_t,\hat{z}_t)$ uniformly in $[0,T]$, i.e.,
\be\nonumber
\lim\limits_{n\rightarrow \infty}\sup\limits_{t\in[0,T]}\Big\{\big|\hat{X}^{(k_n)}_t-\hat{X}_t\big|+\big|\hat{L}^{(k_n)}_t-\hat{L}_t\big|+\big|\hat{z}^{(k_n)}_t-\hat{z}_t\big|\Big\}=0.
\ee
For $0\leq j\leq 2^{k}-1$, $t\in[jT/2^k,(j+1)T/2^k]$, we define
\begin{align*}
(A^{k}f)(t,X,L)=&\sum\limits_{i=1}^Nb^{\text{T}}\big(X^{i}_{jT/2^k},\pi^{N}_{jT/2^k}\big)\frac{\partial f}{\partial x_i}\left(X_t,L_t\right)+\frac{1}{2}\sum\limits_{i=1}^N\Tr\Big(\sigma\sigma^{\text{T}}\big(X^{i}_{jT/2^k},\pi^{N}_{jT/2^k}\big)\frac{\partial^2f}{\partial x_i^2}\left(X_t,L_t\right)\Big)\nonumber\\
& + \frac{1}{2}\sum\limits_{i=1}^N\left|L_t^ih(X^i_t,\pi^N_t)\right|^2\frac{\partial^2f}{\partial l_i^2}\left(X_t,L_t\right).\nonumber
\end{align*}
Then by Theorem \ref{theorem2}, for $s<t$, $\forall \,\mathcal{B}_s(C(\mathbb{R}^{2N}))-$measureable bounded continuous function $F_s$, we have
\be\nonumber
\E^{\hat{\Q}}\bigg[\bigg(f\Big(\hat{X}^{(k_n)}_t,\hat{L}^{(k_n)}_t\Big)-f\Big(\hat{X}^{(k_n)}_s,\hat{L}^{(k_n)}_s\Big)-\int_s^t(A^{k_n}f)\Big(r,\hat{X}^{(k_n)},\hat{L}^{(k_n)}\Big)\D r\Big)F_s\Big(\hat{X}^{(k_n)},\hat{L}^{(k_n)}\Big)\bigg]=0.
\ee

Let $n\rightarrow \infty$, we only need to prove
\be\label{existence1}
\lim\limits_{n\rightarrow \infty}\E^{\hat{\Q}}\bigg[\int_s^t(A^{k_n}f)\Big(r,\hat{X}^{(k_n)},\hat{L}^{(k_n)}\Big)\D rF_s\Big(\hat{X}^{k_n},\hat{L}^{(k_n)}\Big)\bigg]=\E^{\hat{\Q}}\bigg[\int_s^t(Af)\big(r,\hat{X},\hat{L}\big)\D rF_s\big(\hat{X},\hat{L}\big)\bigg].
\ee
If the equation (\ref{existence1}) is true, it follows that
\be\nonumber
\E^{\hat{\Q}}\bigg[\bigg(f\big(\hat{X}_t,\hat{L}_t\big)-f\big(\hat{X}_s,\hat{L}_s\big)-\int_s^t(Af)\big(r,\hat{X},\hat{L}\big)\D r\bigg)F_s\big(\hat{X},\hat{L}\big)\bigg]=0.
\ee
Hence, there exist independent standard Brownian motions $\{\hat{W}^i\}_{i=1}^N$ and $\hat{Y}$, such that 
$$
\left(\hat{\Omega},\hat{\mathcal{F}},\hat{\mathbb{P}},\hat{X},\hat{L},\big\{\hat{W}^i\big\}_{i=1}^N,\hat{Y}\right)
$$ 
is a weak solution of SDE (\ref{particle}).

Now, we begin to prove (\ref{existence1}). Due to the uniform boundedness of 
$$
\int_s^t(A^{k_n}f)\Big(r,\hat{X}^{(k_n)},\hat{L}^{(k_n)}\Big)\D r,
$$ 
it holds that
\begin{align}\label{existence2}
& \displaystyle \lim\limits_{n\rightarrow \infty}\E^{\hat{\Q}}\bigg[\int_s^t(A^{k_n}f)\Big(r,\hat{X}^{(k_n)},\hat{L}^{(k_n)}\Big)\D rF\Big(\hat{X}^{(k_n)},\hat{L}^{(k_n)}\Big)\bigg]\nonumber\\
=& \displaystyle \lim\limits_{n\rightarrow \infty}\E^{\hat{\Q}}\bigg[\int_s^t(A^{k_n}f)\Big(r,\hat{X}^{(k_n)},\hat{L}^{(k_n)}\Big)\D rF\big(\hat{X},\hat{Y}\big)\bigg].
\end{align}

Denote 
$$
\hat{\pi}_t^{N,n}=\sum\limits_{i=1}^N\hat{w}_t^{i,n}\delta_{\hat{X}_t^{i,(n)}}\quad\quad\text{and}\quad\quad\hat{\pi}_t^{N,n}=\sum\limits_{i=1}^N\hat{w}_t^{i}\delta_{\hat{X}_t^{i}}.
$$ 
where $\hat{w}_t^{i,n}=\exp(\hat{z}_t^{i,n})$ and $\hat{w}_t^{i}=\exp(\hat{z}_t^{i})$. By the boundedness of $h,|F|,|\frac{\partial f}{\partial x_i},|\frac{\partial^2 f}{\partial x_i^2},|\frac{\partial^2 f}{\partial l_i^2}|$, we can get
\begin{align}\label{existence3}
&\E^{\hat{\Q}}\bigg[\bigg|\int_s^t(A^{k_n}f)\left(r,\hat{X}^{(k_n)},\hat{L}^{(k_n)}\right)\D rF_s\big(\hat{X},\hat{L}\big)-\int_s^t(Af)\big(r,\hat{X},\hat{L}\big)\D rF_s\big(\hat{X},\hat{L}\big)\bigg|\bigg]\nonumber\\
&\leq  C\,\E^{\hat{\Q}}\bigg[\int_s^t\bigg|(A^{k_n}f)\left(r,\hat{X}^{(k_n)},\hat{L}^{(k_n)}\right)-(Af)\big(r,\hat{X},\hat{L}\big)\bigg|\D r\bigg]\nonumber\\
 &\leq  C\,\sum\limits_{i=1}^N\E^{\hat{\Q}}\bigg[\int_s^t\Big|
 (b ,\sigma \sigma^{\text{T}})(\hat{X}_r^{i,k_n},\hat{\pi}_r^{N,k_n})- (b ,\sigma \sigma^{\text{T}})(\hat{X}^{i}_r,\hat{\pi}_r^N)\Big|\D r\bigg]\nonumber\\
 &+ C\,\sum\limits_{i=1}^N\E^{\hat{\Q}}\left[\int_s^t\Big|
 |h|^2(\hat{X}_r^{i,k_n},\hat{\pi}_r^{N,k_n})- |h|^2(\hat{X}^{i}_r,\hat{\pi}_r^N)\Big|\Big|\hat{L}_r^i\Big|^2+\Big|\big|\hat{L}_r^{i,k_n}\big|^2-\big|\hat{L}_r^i\big|^2\Big|\D r\right]\nonumber\\
&+ C\sup\limits_{0\leq j\leq 2^{k_n}-1} \bigg\{ \sup\limits_{r\in\left[{jT/2^{k_n}},{(j+1)T/2^{k_n}}\right]}\sum\limits_{i=1}^N\E^{\hat{\Q}}\bigg[\Big|
(b ,\sigma \sigma^{\text{T}})(\hat{X}_r^{i,k_n},\hat{\pi}_r^{N,k_n})- (b ,\sigma \sigma^{\text{T}})(\hat{X}^{i,k_n}_{jT/2^{k_n}},\hat{\pi}_{jT/2^{k_n}}^{N,k_n})\Big|\bigg]\bigg\} \nonumber\\
&+C\,\sum\limits_{i=1}^N\E^{\hat{\Q}}\bigg[\int_s^t \bigg| \left(\frac{\partial f}{\partial x_i},\frac{\partial^2f}{\partial x_i^2},\frac{\partial^2f}{\partial l_i^2}\right)\big(\hat{X}^{(k_n)}_r,\hat{L}^{(k_n)}_r\big)-\left(\frac{\partial f}{\partial x_i},\frac{\partial^2f}{\partial x_i^2},\frac{\partial^2f}{\partial l_i^2}\right)\big(\hat{X}_r,\hat{L}_r\big)\bigg|\D r  \bigg ]\nonumber\\
&:=J_1+J_2+J_3+J_4,
\end{align}
where 
$$
\left|\left(b,\sigma \sigma^{\text{T}}\right)(x,y)\right|:=\left|b(x,y)\right|+\left|\sigma \sigma^{\text{T}}(x,y)\right|, 
$$
and
$$
\left| \left(\frac{\partial f}{\partial x_i},\frac{\partial^2f}{\partial x_i^2},\frac{\partial^2f}{\partial l_i^2}\right)(x,y)\right|:= \left| \frac{\partial f}{\partial x_i}(x,y)\right|+\left| \frac{\partial^2 f}{\partial x_i^2}(x,y)\right|+\left| \frac{\partial^2 f}{\partial l_i^2}(x,y)\right|.
$$

Firstly, we prove the convergence of 
$$
J_3=\sup\limits_{0\leq j\leq 2^{k_n}-1} \Big\{ \sup\limits_{r\in\left[{jT/2^{k_n}},{(j+1)T/2^{k_n}}\right]} \sum\limits_{i=1}^N\E^{\hat{\Q}}\Big[\big|(b ,\sigma \sigma^{\text{T}})(\hat{X}_r^{i,k_n},\hat{\pi}_r^{N,k_n})- (b ,\sigma \sigma^{\text{T}})(\hat{X}^{i,k_n}_{jT/2^{k_n}},\hat{\pi}_{jT/2^{k_n}}^{N,k_n})\big|\Big] \Big\}.
$$ 
According to the Lipschitz continuity of $b$ and $\sigma$, it follows that
\begin{align}\label{existence3.5}
&\sum\limits_{i=1}^N\E^{\hat{\Q}}  \Big[\big| (b,\sigma \sigma^{\text{T}}) (\hat{X}_r^{i,k_n},\hat{\pi}_r^{N,k_n} ) - (b,\sigma \sigma^{\text{T}}) ( \hat{X}^{i,k_n}_{jT/2^{k_n}},\hat{\pi}_{jT/2^{k_n}}^{N,k_n} ) \big| \Big] \nonumber\\
\leq&\sum\limits_{i=1}^N C\E^{\hat{\Q}} \bigg[ \Big| \hat{X}_r^{i,k_n}-\hat{X}^{i,k_n}_{jT/2^{k_n}}  \Big|+W_1\big( \hat{\pi}_r^{N,k_n},\hat{\pi}_{jT/2^{k_n}}^{N,k_n} \big) \bigg].
\end{align} 
By the boundedness of $b,\sigma$, we have the following standard estimates with the help of BDG inequality:      
\be\label{lemma4.2}
\E^{\Q}\Big[\max\limits_{t\in[0,T]}\big|\hat{X}_t^{i,k_n}\big|^4\Big]\leq C\quad\text{and}\quad\E^{\Q}\Big[\big|\hat{X}_t^{i,k_n}-\hat{X}_s^{i,k_n}\big|^4\Big]\leq C|t-s|^2,\ \ 1\leq i\leq N.
\ee
Notice $\forall\, t\in({jT/2^k},{(j+1)T/2^k})$,
\be
dw_t^{i,k}=w_t^{i,k}H(X_{jT/2^k}^{i,k},\pi^N_{jT/2^k})\D t+w_t^{i,k}M^{\text{T}}(X_{jT/2^k}^{i,k},\pi^N_{jT/2^k})\D Y_t.\nonumber
\ee
By the boundedness of $(H,M)$ and BDG inequality, we have
\be\label{boundw}
\E^{\Q}\Big[\big|\hat{w}_t^{i,k_n}\big|^4\Big]=\E^{\Q}\Big[\big|w_t^{i,k_n}\big|^4\Big]\leq \frac{C}{N^4},\ \ 1\leq i\leq N.
\ee
Furthermore, by BDG inequality and Cauchy's inequality, it holds that
\begin{equation}\label{lemma4.3}
\E^{\Q}\Big[ \big| \hat{w}_t^{i,k_n}- \hat{w}_s^{i,k_n} \big|^4 \Big] = \E^{\Q}\Big[\big|w_t^{i,k_n}-w_s^{i,k_n}\big|^4\Big] \leq C|t-s|^2\E^{\Q}\Big[\max\limits_{s\leq r\leq t} \big|\hat{w}_r^{i,k_n}\big|^4\Big]\leq\frac{C|t-s|^2}{N^4}.
\end{equation}
In view of (\ref{lemma4.2}), (\ref{boundw}), (\ref{lemma4.3}) and Lemma \ref{lemma4}, for $ r\in[{jT/2^{k_n}},(j+1)T/2^{k_n}]$, we have 
\begin{align}  
&\sum\limits_{i=1}^N\E^{\hat{\Q}}\Big[\big|(b,\sigma \sigma^{\text{T}})(\hat{X}_r^{i,k_n},\hat{\pi}_r^{N,k_n})-(b,\sigma \sigma^{\text{T}})(\hat{X}^{i,k_n}_{jT/2^{k_n}},\hat{\pi}_{jT/2^{k_n}}^{N,k_n})\big|\Big]\nonumber\\
\leq& C\sum\limits_{i=1}^N\E^{\hat{\Q}}\Big[\big|\hat{X}_r^{i,k_n}-\hat{X}^{i,k_n}_{jT/2^{k_n}}\big|^2\Big]^{\frac{1}{2}}+C\E^{\hat{\Q}}\Big[W_1\big(\hat{\pi}_r^{N,k_n},\hat{\pi}_{jT/2^{k_n}}^{N,k_n}\big)^2\Big]^{\frac{1}{2}} \leq C\sqrt{\frac{1}{2^{k_n}}}.\nonumber
\end{align}
Hence, it holds that
\be\nonumber
\lim\limits_{n\rightarrow \infty}J_3=0.
\ee

Next, we prove the convergence of $J_1$ and $J_2$.
In view of Lemma \ref{lemma1}, we have
\be\label{error}
W_1\big(\hat{\pi}_r^{N,k_n},\hat{\pi}_r^N\big)\leq \sum\limits_{i=1}^N\hat{w}_r^{i,k_n}\vee\hat{w}_r^{i}\big|\hat{z}_r^{i,k_n}-\hat{z}_r^i\big|\big|\hat{X}_r^i\big|+\sum\limits_{i=1}^N \hat{w}_r^{i,k_n}\big|\hat{X}_r^{i,k_n}-\tilde{X}_r^i\big|.
\ee
By the similar calculation that deduces (\ref{lemma4.2}) and (\ref{boundw}), we have
\be\nonumber
\E^{\Q}\Big[\big|\hat{w}_t^{i}\big|^8+\big|\hat{w}_t^{i,k_n}\big|^8+\big|\hat{X}_t^{i}\big|^8+\big|\hat{X}_t^{i,k_n}\big|^8\Big]\leq C,\quad \forall 1\leq i\leq N,\,n\geq 1.
\ee
By Cauchy's inequality, it follows from (\ref{error}) that, 
\begin{align}
\nonumber\E^{\hat{\Q}}\Big[W_1\big(\pi_r^{N,k_n},\pi_r^N\big)^{2}\Big]\leq &\sum\limits_{i,j=1}^N\E^{\hat{\Q}}\Big[\big|\hat{z}_r^{i,k_n}-\hat{z}_r^i\big|^{4}\Big]^{\frac{1}{2}}\E^{\hat{\Q}}\Big[\Big(\big|\hat{w}_r^{i,k_n}\big|^4+\big|\hat{w}_r^{i}\big|^4\Big)\big|\hat{X}_r^i\big|^4\Big]^{\frac{1}{2}}\\
\nonumber &+ \sum\limits_{i,j=1}^N  \E^{\hat{\Q}} \Big[ \big|\hat{X}_r^{i,k_n}-\tilde{X}_r^i \big|^2 \Big]^{\frac{1}{2}} \E^{\hat{\Q}} \Big[ \big| \hat{w}_r^{i,k_n} \big|^2 \Big]^{\frac{1}{2}}\nonumber\\
\leq &C \bigg(\sum\limits_{i=1}^N\E^{\hat{\Q}}\Big[\big|\hat{z}_r^{i,k_n}-\hat{z}_r^i\big|^4\Big]^{\frac{1}{2}}+\sum\limits_{i=1}^N\E^{\hat{\Q}}\Big[\big|\hat{X}_r^{i,k_n}-\hat{X}_r^i\big|^4\Big]^{\frac{1}{2}}\bigg).\nonumber
\end{align}  
Denote 
$$
T^{i,n}=\max\limits_{0\leq r\leq t}\left|\hat{X}_r^{i,k_n}-\hat{X}_r^{i}\right|^4+\max\limits_{0\leq r\leq t}\left|\hat{z}_r^{i,k_n}-\hat{z}_r^{i}\right|^4+\max\limits_{0\leq r\leq t}\left|\hat{L}_r^{i,k_n}-\hat{L}_r^{i}\right|^2.
$$
Due to the uniform convergence of $\{(\hat{X}_r^{i,k_n},\hat{L}_r^{i,k_n},\hat{z}_r^{i,k_n})\}_{n\geq 1}$, we can observe that $T^{i,n}$ convergence to 0  a.s. and $T^{i,n}$ is uniformly integrable. So we can get that 
\begin{align}
\sum\limits_{i=1}^N\E^{\hat{\Q}}\Big[\max\limits_{s\leq r\leq t}\big|\hat{X}_r^{i,k_n}-\hat{X}_r^{i}\big|\Big]\leq  \sum\limits_{i=1}^N\E^{\hat{\Q}}\Big[\max\limits_{s\leq r\leq t}\big|\hat{X}_r^{i,k_n}-\hat{X}_r^{i}\big|^4\Big]^{\frac{1}{4}} \rightarrow 0 \nonumber ,
\end{align}
and
\be\label{co1}
\max\limits_{0\leq r\leq t}\E^{\hat{\Q}}\Big[W_1\big(\pi_r^{N,k_n},\pi_r^N\big)\Big]\leq C\max\limits_{0\leq r\leq t}\E^{\hat{\Q}}\Big[W_1\big(\pi_r^{N,k_n},\pi_r^N\big)^2\Big]^{\frac{1}{2}} \rightarrow 0.
\ee

By the similar calculation that deduces (\ref{existence3.5}), we obtain
\begin{align}\label{existence6}  
&\int_s^t\sum\limits_{i=1}^N\E^{\hat{\Q}}\Big[\big|(b,\sigma \sigma^{\text{T}})(\hat{X}_r^{i,k_n},\hat{\pi}_r^{N,k_n})-(b,\sigma \sigma^{\text{T}})(\hat{X}_r^{i},\hat{\pi}_r^{N})\big|\Big]\D r\nonumber\\
\leq& \, C\sum\limits_{i=1}^N\E^{\hat{\Q}}\bigg[\int_s^t\Big|\hat{X}_r^{i,k_n}-\hat{X}_r^{i}\Big|\D r\bigg]+C\,\E^{\hat{\Q}}\bigg[\int_s^t W_1\big(\hat{\pi}_r^{N,k_n},\hat{\pi}_r^N\big)\D r\bigg]\nonumber\\
\leq &\, C\sum\limits_{i=1}^N\E^{\hat{\Q}}\left[\max\limits_{s\leq r\leq t}\left|\hat{X}_r^{i,k_n}-\hat{X}_r^{i}\right|\right]+C\max\limits_{s\leq r\leq t}\E^{\hat{\Q}}\Big[W_1\big(\hat{\pi}_r^{N,k_n},\hat{\pi}_r^N\big)\Big].
\end{align}
It follows from (\ref{existence6}) that
\begin{align}\nonumber
J_1=\int_s^t\sum\limits_{i=1}^N\E^{\hat{\Q}}\Big[\big|b(\hat{X}_r^{i,k_n},\pi_r^{N,k_n})-b(\hat{X}_r^{i},\pi_r^{N})\big|+\big|\sigma \sigma^{\text{T}}(\hat{X}_r^{i,k_n},\pi_r^{N,k_n})-\sigma \sigma^{\text{T}}(\hat{X}_r^{i},\pi_r^{N})\big|\Big]\D r\rightarrow 0.
\end{align}
Similarly, by the integrable of $|L_t^i|^4$ and boundedness of $h$, we have
\begin{align}
J_2=& \int_s^t\sum\limits_{i=1}^N\E^{\hat{\Q}}\Big[\left||h|^2(\hat{X}_r^{i,k_n},\hat{\pi}_r^{N,k_n})- |h|^2(\hat{X}^{i}_r,\hat{\pi}_r^N)\right|\Big|\hat{L}_r^i\Big|^2+\Big|\big|\hat{L}_r^{i,k_n}\big|^2-\big|\hat{L}_r^i\big|^2\Big|\Big]\D r\nonumber\\
\leq&\int_s^t\sum\limits_{i=1}^N\E^{\hat{\Q}}\Big[\big|
 h(\hat{X}_r^{i,k_n},\hat{\pi}_r^{N,k_n})- h(\hat{X}^{i}_r,\hat{\pi}_r^N)\big|^2\Big]^{\frac{1}{2}}+\E^{\hat{\Q}}\Big[\big|\hat{L}_r^{i,k_n}-\hat{L}_r^i\big|^2\Big]^{\frac{1}{2}}\D r\nonumber\\
\leq& C\bigg\{\sum\limits_{i=1}^N\E^{\hat{\Q}}\Big[\max\limits_{s\leq r\leq t}\big|\hat{X}_r^{i,k_n}-\hat{X}_r^{i}\big|^2 + \max\limits_{s\leq r\leq t}\big|\hat{L}_r^{i,k_n}-\hat{L}_r^i\big|^2 \Big]^{\frac{1}{2}}+\max\limits_{s\leq r\leq t}\E^{\hat{\Q}}\Big[W_1\big(\hat{\pi}_r^{N,k_n},\hat{\pi}_r^N\big)^2\Big]^{\frac{1}{2}} \bigg\}.\nonumber
\end{align}
By the uniform convergence of $T^{i,k_n}$ and (\ref{co1}), we can deduce the convergence of $J_2$.

Finally, the weak existence of the solution to SDE (\ref{particle}) can be proved if we show the convergence of $J_4$, which is clear if we notice the uniform convergence of $(\hat{X}^{(k_n)},\hat{L}^{k_n})$ and $f\in C_b^2(\mathbb{R}^{2N})$.
\end{proof}

By the results of Theorem \ref{theorem1} and Theorem \ref{theorem4.3}, we can complete the proof of Theorem \ref{Theorem exist-unique} by Yamada-Watanabe Theorem.

\subsection{Proof of particle approximation}
The following lemma guarantees the well-posedness of the limiting process.
\begin{lemma}\label{well-poseness-MV}
There exists a probability space $({\Omega},\mathcal{{F}},\mathbb{{Q}},\{\mathcal{{F}}_t\}, \{{W}_t^i\}_{i=1}^\infty, {Y} )$ such that the following mean-field SDE:
\be\label{sde-Vlasov}
\begin{cases}
\D \bar{X}_t^i=b\left(\bar{X}_t^i,\pi_t\right)\D t+\sigma\left(\bar{X}_t^i,\pi_t\right)\D W_t^i,\qquad\bar{X}_0^i=x,\\[0.3cm]
\D \bar{L}_t^i=h(\bar{X}_t^i,\pi_t)\bar{L}_t^i\D Y_t,\qquad\bar{L}_0^i=1 ,\ \ \ i=1,2,...,N,\ \ \ t\in[0,T],
\end{cases}
\ee
has a unique solution $\{ (\bar{X}^{i},\bar{L}^{i} )\}_{i=1}^{\infty}$ in this probability space, where $\{W_t^i\}_{i=1}^\infty, Y$ are the independent Brownian motions in this probability space with the natural filtration $\mathcal{F}_t=\sigma\{ ( W^i_s, Y_s ):s\leq t,\,\,1\leq i\}$,  and $\pi_t$ is the random probability measure such that for all $i\geq 1$,
\begin{equation}\label{2312251}
\int_{\mathbb{R}}\phi(x)\mathrm{d}\pi_t=\frac{\E^{\Q}\left[\bar{L}_t^i\phi(\bar{X}_t^i)|\mathcal{F}_t^Y\right]}{\E^{\Q}\left[\bar{L}_t^i|\mathcal{F}_t^Y\right]},\qquad\forall\,\phi\in C^{\infty}_c(\mathbb{R}^n).
\end{equation}
\end{lemma}
The proof of lemma \ref{well-poseness-MV} is not difficult to obtain by similar arguments in \cite{Buckdahn} or \cite{MR1705602}. We give some remarks here about the limiting process (\ref{sde-Vlasov}). In fact, the solution $\{(\bar{X}_t^i,\bar{L}_t^i)\}_{i=1}^\infty$ of (\ref{sde-Vlasov}) has the same distribution under $\mathbb{P}(\cdot|\mathcal{F}_t^Y)$. Therefore $\forall\, i\geq 1$, we have
$$
\int_{\mathbb{R}}\phi(x)\mathrm{d}\pi_t=\frac{\E^{\Q}\left[\bar{L}_t^i\phi(\bar{X}_t^i)|\mathcal{F}_t^Y\right]}{\E^{\Q}\left[\bar{L}_t^i|\mathcal{F}_t^Y\right]},\quad\forall\,\phi\in C^{\infty}_c(\mathbb{R}^n).
$$
This symmetry is used in our following proof. 

What's more, we point out that we can construct the weak solution of the partially observed system from the solution of (\ref{sde-Vlasov}). Given $i\geq 1$, we denote 
$$
\frac{\D \mathbb{P}^i}{\D \Q}=L_T^i.
$$ 
Using Girsanov theorem, we can easily check that $(\bar{X}^i,Y)$ is the solution to
\be\label{MV-2}
\begin{cases}
\D \bar{X}_t^i=b\left(\bar{X}_t^i,\pi_t\right)\D t+\sigma\left(\bar{X}_t^i,\pi_t\right)\D W_t^i,\quad\bar{X}_0^i=x,\\  
\D Y_t=h\left(\bar{X}_t^i,\pi_t\right)\D t +\D B_t, \quad Y_0=0 ,\ \ \ t\in[0,T],
\end{cases}
\ee
where $W_t^i$, $B_t$ are independent standard Brownian motion under $(\Omega, \mathcal{F},\mathbb{P}^i)$. By Kallianpur-Striebel formula, we have
$$
\int_{\mathbb{R}}\phi(x)\mathrm{d}\pi_t=\E^{\mathbb{P}^i}\left[\phi(\bar{X}_t^i)|\mathcal{F}_t^Y\right],\quad\forall\,\phi\in C^{\infty}_c(\mathbb{R}).
$$
In other words, $(\Omega, \mathcal{F},\mathbb{P}^i,\{\mathcal{F}_t\},\bar{X}_t^i,Y_t,W_t^i,B_t)$ is a weak solution of the conditional McKean-Vlasov SDE (\ref{MV-2}).

Now we begin our proof of particle approximation.
For a given $N$, we simply denote $\{(X^{i},L^{i})\}_{i=1}^N$ instead of $\{(X^{i,N},L^{i,N})\}_{i=1}^N$. Denote 
$$
\Delta \bar{X}_t^i=X_t^i-\bar{X}_t^i,\quad \Delta \bar{z}_t^i=z_t^i-\bar{z}_t^i,\quad \Delta \sigma_t^i=\sigma(X_t^i,\pi_t^N)-\sigma(\bar{X}_t^i,\pi_t),\quad \Delta \bar{M}_t^i=M(X_t^i,\pi_t^N)-M(\bar{X}_t^i,\pi_t).
$$

\begin{lemma}   \label{Wasser-empirical-bound}
\begin{align*} \max_{t\in[0,T]}\E^{\Q}\Big[W_1\big(\pi_t^N,\bar{\pi}_t^N\big)^2\Psi_t\Big]\leq CN^{-\frac{8}{7n+8}},
\end{align*}
where $C=C(K,T,n)$ is the constant.
\end{lemma}
\begin{proof}
    With the similar calculation in (\ref{u-5})-(\ref{unique-16}), we obtain
    \begin{align}\label{Westi}
W_1\Big(\pi_t^N,\bar{\pi}_t^N\Big)^2\Psi_t\leq 2\sum\limits_{j,k=1}^N\big| X_t^j\big|^2\Big(\big|w_t^j\big|^2+\big|\bar{w}_t^j\big|^2\Big)\big|\Delta \bar{z}_t^k\big|^2\Psi_t+2\sum\limits_{j,k=1}^N \big|\bar{w}_t^j\big|^2\big|\Delta \bar{X}_t^k\big|^2\Psi_t.
    \end{align}
By Ito's formula, 
\begin{align*}
   &\mathrm{d} \Big(1+\big|X_t^j\big|^2\Big)\Big(\big|\bar{w}_t^j\big|^2+\big|w_t^j\big|^2\Big)\Big(\big|\Delta \bar{z}_t^k\big|^2+\big|\Delta \bar{X}_t^k\big|^2\Big)\Psi_t\\
   = &2\Psi_t\Big(\big|\bar{w}_t^j\big|^2+\big|w_t^j\big|^2\Big)\Big(\big|\Delta \bar{z}_t^k\big|^2+\big|\Delta \bar{X}_t^k\big|^2\Big)\big(X_t^j\big)^\text{T}b_t^j\mathrm{d}t\\
&+2\Psi_t\Big(1+\big|X_t^j\big|^2\Big)\Big(\big|w_t^j\big|^2H_t^j+\big|\bar{w}_t^j\big|^2\bar{H}_t^j\Big)\Big(\big|\Delta \bar{z}_t^k\big|^2+\big|\Delta \bar{X}_t^k\big|^2\Big)\mathrm{d}t\\
 &+2\Psi_t\Big(1+\big|X_t^j\big|^2\Big)\Big(\big|\bar{w}_t^j\big|^2+\big|w_t^j\big|^2\Big)\Big(\Delta \bar{z}_t^k\Delta \bar{R}_t^k+ \big(\Delta\bar{X}_t^k\big)^T\Delta \bar{b}_t^k\Big)\mathrm{d}t\\
   &+\Psi_t\Big(\big|\bar{w}_t^j\big|^2+\big|w_t^j\big|^2\Big)\Big(\big|\Delta \bar{z}_t^k\big|^2+\big|\Delta \bar{X}_t^k\big|^2\Big)\big|\sigma_t^j\big|^2\mathrm{d}t\\   &+\Psi_t\Big(1+\big|X_t^j\big|^2\Big)\big|w_t^jM_t^j+\bar{w}_t^j\bar{M}_t^j\big|^2\Big(\big|\Delta \bar{z}_t^k\big|^2+\big|\Delta \bar{X}_t^k\big|^2\Big)\mathrm{d}t\nonumber\\
   &+4\Psi_t\Delta \bar{z}_t^k\Big(1+\big|X_t^j\big|^2\Big)\Big(w_t^jM_t^j+\bar{w}_t^j\bar{M}_t^j\big)^{\text{T}}\Delta \bar{M}_t^k\mathrm{d}t\nonumber\\
   &+\Big(1+\big|X_t^j\big|^2\Big)\Big(\big|\bar{w}_t^j\big|^2+\big|w_t^j\big|^2\Big)\Big(\big|\Delta \bar{z}_t^k\big|^2+\big|\Delta \bar{X}_t^k\big|^2\Big)\mathrm{d}\Psi_t+\mathrm{d}V_t,
\end{align*}
where $V$ is a square-integrable martingale.

Denote 
$$
B_t^{j}=\Big(1+\big|X_t^j\big|^2\Big)\Big(\big|\bar{w}_t^j\big|^2+\big|w_t^j\big|^2\Big)\Psi_t.
$$
By BDG inequality, 
\begin{align}\label{boundB}
   \E^{\Q}\Big[\max_{t\in[0,T]}\big|B_t^j\big|^2\Big]\leq \E^{\Q}\Big[\max_{t\in[0,T]}\Big(1+\big|X_t^j\big|^2\Big)^4\Big]^{\frac{1}{2}}\E^{\Q}\Big[\max_{t\in[0,T]}\Big(\big|\bar{w}_t^j\big|^2+\big|w_t^j\big|^2\Big)^4\Big]^{\frac{1}{2}}\leq \frac{C}{N^4}.
\end{align}
By the Lipschitz continuity and boundedness of $b,\sigma,M,H$, 
\begin{align}    \label{431}
&\E^{\Q}\Big[B_t^{j}\Big(\big|\Delta \bar{z}_t^k\big|^2+\big|\Delta \bar{X}_t^k\big|^2\Big)\Big]\nonumber \\
\leq &C\int_0^t \E^{\Q}\Big[B_s^j\Big(\big|\Delta \bar{z}_s^k\big|^2+\big|\Delta \bar{X}_s^k\big|^2\Big)\Big]\mathrm{d}s+C\int_0^t \E^{\Q}\Big[B_s^jW_1\big(\pi_s^N,\pi_s\big)^2\Big]\mathrm{d}s\nonumber\\
&-\alpha N\int_0^t \E^{\Q}\bigg[\sum\limits_{i=1}^N \bigg[\frac{1}{N^2}+\Big(\big|w_s^i\big|^2+\big|\bar{w}_s^i\big|^2\Big)\Big(1+\big|\tilde{X}_s^i\big|^2\Big)\bigg] \Big(\big|\Delta \bar{z}_s^k\big|^2+\big|\Delta \bar{X}_s^k\big|^2\Big)B_s^j\bigg]\mathrm{d}s.
\end{align}

By the triangle inequality of Wasserstein distance, 
\begin{align}\label{Wes2}
     W_1\left(\pi_s^N,\pi_s\right)^2\leq 2W_1\left(\pi_s^N,\bar{\pi}^N_s\right)^2+2W_1\left(\bar{\pi}_s^N,\pi_s\right)^2.
     \end{align}
Taking $\alpha$ large enough and combing (\ref{Westi}, (\ref{431}) and (\ref{Wes2}), we obtain 
\begin{align*}
    \E^{\Q}\Big[\sum\limits_{j,k=1}^N B_t^{j}\Big(\big|\Delta \bar{z}_t^k\big|^2+\big|\Delta \bar{X}_t^k\big|^2\Big)\Big]&\leq C\int_0^t N\E^{\Q}\Big[\sum_{j=1}^NB_s^jW_1\big(\bar{\pi}_s^N,\pi_s\big)^2\Big]\mathrm{d}s  \\
   &\leq CN\sum_{j=1}^N\E^{\Q}\Big[\max_{s\in[0,t]}\big|B_s^j\big|^2\Big]^{\frac{1}{2}}\E^{\Q}\Big[\max_{s\in[0,t]}W_1\big(\bar{\pi}_s^N,\pi_s\big)^4\Big]^{\frac{1}{2}} .
\end{align*}

Then it follows from  (\ref{Westi}), (\ref{boundB}) and Lemma \ref{convergence rate} that 
\begin{align*}   \E^{\Q}\Big[W_1\Big(\pi_t^N,\bar{\pi}_t^N\Big)^2\Psi_t\Big]\leq  \E^{\Q}\Big[\sum\limits_{j,k=1}^N B_t^{j}\Big(\big|\Delta \bar{z}_t^k\big|^2+\big|\Delta \bar{X}_t^k\big|^2\Big)\Big]\leq CN^{-\frac{8}{7n+8}},
\end{align*}
the proof is completed.
\end{proof}

\begin{proof}[Proof of Theorem \ref{multi}]
According to Lemma \ref{lemma3}, we have
\begin{align*}
       \big|\Delta \bar{X}_t^i\big|^2\Psi_t
\leq  C\int_0^t W_1\left(\pi_s^N,\pi_s\right)^2\Psi_s\D s+C\int_0^t\big|\Delta \bar{X}_s^i\big|^2\Psi_s\D s+C\int_0^t\Psi_s\Big( \Delta \bar{X}_s^{i}\Big)^{\text{T}} \Delta\bar{\sigma}_s^i\D W_s^i .
\end{align*}
By the similar calculation in (\ref{u-4}), (\ref{u-5}),(\ref{unique-16}) and the
we have
\begin{align}\label{422}
&\frac{1}{N}\sum\limits_{i=1}^N\left|\Delta \bar{X}_t^i\right|^2\Psi_t\nonumber\\
\leq & \, C\int_0^t \bigg(1+N\sum\limits_{k=1}^N\big|\bar{w}_s^k\big|^2\bigg)\frac{1}{N}\sum\limits_{i=1}^N\big|\Delta \bar{X}_s^i\big|^2 \Psi_s \D s+\frac{C}{N}\sum_{i=1}^N\int_0^t\Big(\Delta \bar{X}_s^i\Big)^{\text{T}} \Delta\bar{\sigma}_s^i\Psi_s\D W_s^i\nonumber\\
& + C\int_0^t\sum\limits_{k=1}^N \Big(\big|w_s^k\big|^2+\big|\bar{w}_s^k\big|^2\Big)\big|X_s^k\big|^2 \sum\limits_{i=1}^N\big|\Delta \bar{z}_s^i\big|^2\Psi_s\D s+C\int_0^tW_1\big(\pi_s^N,\pi_s\big)^2\Psi_s\mathrm{d}s\nonumber\\
& -\alpha \int_0^t \sum\limits_{k=1}^N \bigg[\frac{1}{N^2}+\Big(\big|w_s^k\big|^2+\big|\bar{w}_s^k\big|^2\Big)\Big(1+\big|X_s^k\big|^2\Big)\bigg]\sum\limits_{i=1}^N\big|\Delta \bar{X}_s^i\big|^2\Psi_s\D s.
\end{align}
Similarly, we obtain that
\begin{align}\label{433}
\frac{1}{N}\sum\limits_{i=1}^N\big| \Delta\bar{z}_t^i\big|^2\Psi_t \leq& \,C \int_0^t  \Big(1+N\sum\limits_{i=1}^N\big|\bar{w}_s^i\big|^2\Big)\frac{1}{N}\sum\limits_{i=1}^N\big| \Delta \bar{X}_s^i\big|^2\Psi_s \D s+\frac{C}{N}\sum_{i=1}^N\int_0^t\Delta \bar{z}_s^i\Psi_s\Big(\Delta \bar{M}_s^i\Big)^{\text{T}}\D Y_s\nonumber\\
    &+\big(C-\alpha \big) \int_0^t \sum\limits_{k=1}^N \bigg[\frac{1}{N^2}+\Big(\big|w_s^k\big|^2+\big|\bar{w}_s^k\big|^2\Big)\Big(1+\big|X_s^k\big|^2\Big)\bigg]\sum\limits_{i=1}^N\big|\Delta \bar{z}_s^i\big|^2\Psi_s\D s\nonumber\\
     &+C\int_0^tW_1\big(\pi_s^N,\pi_s\big)^2\Psi_s\D s.
\end{align}
Taking $\alpha$ large enough, we get that
\begin{align}\label{boundxz}
&\frac{1}{N}\sum\limits_{i=1}^N\Big[\sup_{r\in[0,t]}\big|\Delta\bar{X}_r^i\big|^2\Psi_r+\sup_{r\in[0,t]}\big|\Delta \bar{z}_r^i\big|^2\Psi_r\Big]\nonumber\\
\leq & C\sup_{r\in[0,t]}\int_0^r W_1\left(\bar{\pi}_s^N,\pi_s\right)^2 \Psi_s \D s+\frac{C}{N}\sum_{i=1}^N\sup_{r\in[0,t]}\int_0^r\Delta \bar{z}_s^i\Psi_s\Big(\Delta \bar{M}_s^i\Big)^{\text{T}}\D Y_s\nonumber\\
&+ \frac{C}{N}\sum_{i=1}^N\sup_{r\in[0,t]}\int_0^r\Big(\Delta \bar{X}_s^i\Big)^{\text{T}} \Delta\bar{\sigma}_s^i\Psi_s\D W_s^i.
\end{align}

With the help of BDG inequality, we obtain
\begin{align}\label{boundz}
\E^\mathbb{Q}\Big[\sup_{r\in[0,t]}\Big|\int_0^r\Delta \bar{z}_s^i\Psi_s\Big(\Delta \bar{M}_s^i\Big)^{\text{T}}\D Y_s\Big|\Big]\leq &\,C\,\E^\mathbb{Q}\Big[\Big(\int_0^t\big|\Delta \bar{z}_s^i\big|^2\Psi_s^2\big|\Delta \bar{M}_s^i\big|^2\D s\Big)^{\frac{1}{2}}\Big]\nonumber\\
\leq &\,C\,\E^\mathbb{Q}\Big[\Big(\max_{s\in[0,t]}\Big\{\big|\Delta \bar{z}_s^i\big|^2\Psi_s\Big\}\int_0^t\Psi_s\big|\Delta \bar{M}_s^i\big|^2\D s\Big)^{\frac{1}{2}}\Big]\nonumber\\
\leq & C\delta\E^\mathbb{Q}\Big[\max_ {s\in[0,t]}\big|\Delta \bar{z}_s^i\big|^2\Psi_s\Big] +\frac{C}{\delta}\E^\mathbb{Q}\Big[\int_0^t\Psi_s\big|\Delta \bar{M}_s^i\big|^2\D s\Big].
\end{align}
where $\delta>0$ is a constant. Similarly, we have
\begin{align}\label{boundx}
\E^\mathbb{Q}\Big[\sup_{r\in[0,t]}\Big|\int_0^r\Big(\Delta \bar{X}_s^i\Big)^{\text{T}} \Delta\bar{\sigma}_s^i\Psi_s\D W_s^i\Big|\Big]\leq C\delta\E^\mathbb{Q}\Big[\max_ {s\in[0,t]}\big|\Delta \bar{X}_s^i\big|^2\Psi_s\Big] +\frac{C}{\delta}\E^\mathbb{Q}\Big[\int_0^t\Psi_s\big|\Delta \bar{\sigma}_s^i\big|^2\D s\Big].
\end{align}
Combing (\ref{boundxz})-(\ref{boundx}) and choosing $\delta$ small enough, then we have 
\begin{align}
&\frac{1}{N}\sum\limits_{i=1}^N\E^\mathbb{Q}\Big[\sup_{r\in[0,t]}\big|\Delta\bar{X}_r^i\big|^2\Psi_r+\sup_{r\in[0,t]}\big|\Delta \bar{z}_r^i\big|^2\Psi_r\Big]\nonumber\\\nonumber
\leq \,&C\,\E^\mathbb{Q}\Big[\int_0^t W_1\left(\bar{\pi}_s^N,\pi_s\right)^2\Psi_s \D s\Big]+\frac{C}{N}\sum_{i=1}^N\E^\mathbb{Q}\Big[\int_0^t\Psi_s\big|\Delta \bar{\sigma}_s^i\big|^2+\Psi_s\big|\Delta \bar{M}_s^i\big|^2\D s\Big].
\end{align}
In view of the Lipschitz continuity of $(\sigma,M)$ and Gronwall's inequality, we get
\begin{align}
&\frac{1}{N}\sum\limits_{i=1}^N\E^\mathbb{Q}\Big[\sup_{r\in[0,t]}\big|\Delta\bar{X}_r^i\big|^2\Psi_r+\sup_{r\in[0,t]}\big|\Delta \bar{z}_r^i\big|^2\Psi_r\Big]\nonumber\\
\leq &C\,\E^\mathbb{Q}\Big[\int_0^t W_1\left(\bar{\pi}_s^N,\pi_s\right)^2\Psi_s +W_1\left(\bar{\pi}_s^N,\pi_s^N\right)^2\Psi_s\D s\Big]\nonumber\\\nonumber
\leq &C\,\max_{s\in[0,t]}\E^\mathbb{Q}\Big[ W_1\left(\bar{\pi}_s^N,\pi_s\right)^4\Big]^{\frac{1}{2}} +C\max_{s\in[0,t]}\E^\mathbb{Q}\Big[W_1\left(\bar{\pi}_s^N,\pi_s^N\right)^2\Psi_s\Big].
\end{align}

By the same distribution of $\{|X_t^i-\bar{X}_t^i|^2\Psi_t\}_{i=1}^N$ and the result of Lemma \ref{convergence rate}, Lemma \ref{Wasser-empirical-bound}, then we conclude with Theorem \ref{multi}.     
\end{proof}

\begin{proof}[Proof of Corollary \ref{coro2.1}]
Let $\beta>0$ be an undetermined constant. By Theorem \ref{multi}, we can get that
\begin{align}\label{mul_1}
\E^{\Q}\Big[\max_{t\in[0,T]}\big|X_t^{i}-\bar{X}^{i}_t\big|^2\mathbbm{1}_{\left\{\Psi_t\geq e^{-\beta}\right\}}\Big]\leq  e^{\beta}\E^{\Q}\Big[\max_{t\in[0,T]}\big|X_t^{i}-\bar{X}^{i}_t\big|^2\Psi_t\mathbbm{1}_{\{\Psi_t\geq e^{-\beta}\}}\Big]\leq Ce^{\beta}N^{-\frac{8}{7n+8}}.
\end{align}
Consequently, by Cauchy's inequality, it holds that
\begin{align}\label{mul_2}
\E^{\Q}\Big[\max_{t\in[0,T]}\big|X_t^{i}-\bar{X}^{i}_t\big|^2\mathbbm{1}_{\left\{\Psi_t< e^{-\beta}\right\}}\Big] &\leq C\E^{\Q}\Big[\max_{t\in[0,T]}\big(\big|X_t^{i}\big|^2+\big|\bar{X}^{i}_t\big|^2\Big)\mathbbm{1}_{\big\{\Psi_t< e^{-\beta}\big\}}\Big]\nonumber\\
&\leq C\E^{\Q}\Big[\max_{t\in[0,T]}\big|X_t^{i}\big|^{2q}+\max_{t\in[0,T]}\big|\bar{X}^{i}_t\big|^{2q}\Big]^{\frac{1}{q}}\E\Big[\max_{t\in[0,T]}\mathbbm{1}_{\big\{\Psi_t< e^{-\beta}\big\}}\Big]^{\frac{1}{p}}\nonumber\\
&\leq C\Q\Big(\min_{t\in[0,T]}\Psi_t< e^{-\beta}\Big)^{\frac{1}{p}}.
\end{align}
With the help of Chebyshev's inequality, we have
\begin{align*}
\mathbb{Q}\Big[\min_{t\in[0,T]}\Psi_t< e^{-\beta}\Big]&=\mathbb{Q}\bigg[\max_{t\in[0,T]}\Big\{N\sum\limits_{i=1}^N\Big(\big|w_t^i\big|^2+\big|\bar{w}_t^i\big|^2\Big)\big(1+|X_t^i|^2\big)\Big\}>\beta\bigg]\\
&\leq \frac{1}{\beta} \E^{\Q} \Big[\max_{t\in[0,T]} N\sum\limits_{i=1}^N\Big(\big|w_t^i\big|^2+\big|\bar{w}_t^i\big|^2\Big)\Big(1+\big|X_t^i\big|^2\Big)\Big].
\end{align*}
Using BDG inequality and Gronwall's inequality, we have the following standard estimates
\begin{align}\label{bdg}
\E^{\Q}\!\Big[\max\limits_{s\in[0,T]}\Big\{\big|w_s^i\big|^2+\big|\bar{w}_s^i\big|^2\Big\}\Big] + \E^{\Q}\!\Big[\max\limits_{s\in[0,T]}\Big\{\big|w_s^i\big|^4+\big|\bar{w}_s^i\big|^4\Big\}\Big]^{\frac{1}{2}}\leq \frac{C}{N^2},\quad \E^{\Q}\!\Big[\max\limits_{s\in[0,T]}\big|X_s^j\big|^4\Big]^{\frac{1}{2}}\leq C.
\end{align}
Thus we have
\be\label{mul_3}
\mathbb{Q}\Big[\min_{t\in[0,T]}\Psi_t< e^{-\beta}\Big]\leq  \frac{C}{\beta}.
\ee
By the definition of $\tau_\beta$ and (\ref{mul_1}), we get 
$$
\E^{\Q}\Big[\max_{t\in[0,\tau_{\beta}\wedge T]}\big|X_t^{i}-\bar{X}^{i}_t\big|^2\Big]\leq  \E^{\Q}\Big[\max_{t\in[0, T]}\big|X_t^{i}-\bar{X}^{i}_t\big|^2\mathbbm{1}_{\left\{\Psi_t\geq e^{-\beta}\right\}}\Big]\leq Ce^{\beta}N^{-\frac{8}{7n+8}},
$$
and by (\ref{mul_3}), we have
$$
\mathbb{Q}(\tau_{\beta}>T)\geq  1-\frac{C}{\beta}.
$$
Then we prove the first estimate in Corollary~\eqref{coro2.1}.

Next, by inequality (\ref{mul_1}), (\ref{mul_2}) and (\ref{mul_3}), it holds that
\begin{align}
\E^{\Q}\Big[\max_{t\in[0,T]}\big|X_t^{i}-\bar{X}^{i}_t\big|^2\Big]&=\E^{\Q}\Big[\max_{t\in[0,T]}\left|X_t^{i}-\bar{X}^{i}_t\right|^2\mathbbm{1}_{\left\{\Psi_t\leq e^{-\alpha_N}\right\}}\Big]+\E^{\Q}\Big[\max_{t\in[0,T]}\big|X_t^{i}-\bar{X}^{i}_t\big|^2\mathbbm{1}_{\left\{\Psi_t> e^{-\alpha_N}\right\}}\Big]\nonumber\\
&\leq Ce^{\beta}N^{-\frac{8}{7n+8}}+C\beta^{-\frac{1}{2p}}.\nonumber
\end{align}
 If we take $\beta=\Lambda\ln(N)$, where $\Lambda=\Lambda(K,T)$ is a constant small enough, then we have
 $$
 \E^{\Q}\Big[\max_{t\in[0,T]}\big|X_t^{i}-\bar{X}^{i}_t\big|^2\Big]\leq C\big[\ln\big(N\big)\big]^{-\frac{1}{2p}},\quad\quad \forall p>1.
 $$
 Similarly, by Lemma \ref{Wasser-empirical-bound}, we have 
  $$
 \max_{t\in[0,T]}\E^{\Q}\Big[W_1\big(\pi_t^N,\pi_t\big)^2\Big]\leq C\big[\ln\big(N\big)\big]^{-\frac{1}{2p}},\quad\quad \forall p>1,
 $$
 then the proof is completed.
\end{proof}

\subsection{Proof of Theorem \ref{Euler scheme}}
\begin{proof}[Proof of Theorem \ref{Euler scheme}]
We define 
$$
z_t^{i,\Delta}=\ln\big(w_t^{i,\Delta}\big).
$$ 
With the help of Ito's formula, we have
\be\nonumber
\D z_t^{i,\Delta}= R\Big(X_{\theta(t,\Delta)}^{i,\Delta},\pi_{\theta(t,\Delta)}^N\Big)+M^{\text{T}}\Big(X_{\theta(t,\Delta)}^{i,\Delta},\pi_{\theta(t,\Delta)}^N\Big) \D Y_t.
\ee

Denote multiplier 
$$
\Psi_t^{\Delta}=\exp\bigg\{ -\alpha N\sum\limits_{k=1}^N\int_0^t\left[\frac{1}{N^2}+\Big(\big|w_s^k\big|^2+\big|w_s^{k,\Delta}\big|^2\Big)\Big(1+\big|X_s^{k}\big|^2\Big)\right]\mathrm{d}s \bigg\}.
$$ 
By the boundedness of $b$ and $\sigma$, we have
\begin{align}\label{dis1}
\E^\Q\Big[ \big|X_s^i-X_{\theta(s,\Delta)}^{i,\Delta}\big|^2\Big]&\leq2\E^\Q\Big[ \big|X_s^i-X_s^{i,\Delta}\big|^2\Big]+2\E^\Q\Big[ \big|X_{\theta(s,\Delta)}^{i,\Delta}-X_s^{i,\Delta}\big|^2\Big]\nonumber\\
 &\leq 2\E^\Q\left[ \left|X_s^i-X_s^{i,\Delta}\right|^2\right]+C\Delta.
\end{align}
With the similar calculation that deduces Lemma \ref{lemma4}, we have
\be    \nonumber\E^\Q\Big[W_1\big(\pi_{s}^{N,\Delta},\pi_{\theta(s,\Delta)}^{N,\Delta}\big)^2\Big]\leq C\Delta.
\ee
Similarly, it holds that
\be\label{dis2}
\E^\Q\left[W_1\left(\pi_{s}^{N},\pi_{\theta(s,\Delta)}^{N,\Delta}\right)^2\right]\leq 2\E^\Q\left[W_1\left(\pi_{s}^{N},\pi_{s}^{N,\Delta}\right)^2\right]+C\Delta.
\ee
Combing  (\ref{dis1}) and (\ref{dis2}), we have
\begin{align*}
\E^\Q\Big[\Big|f\big(X_t^i,\pi_{t}^{N}\big)-f\big(X_{\theta(t,\Delta)}^{i,\Delta},\pi_{\theta(t,\Delta)}^{N,\Delta}\big)\Big|^2\Big]\leq&\displaystyle C\, \E^\Q\Big[ \big|X_t^i-X_{\theta(t,\Delta)}^{i,\Delta}\big|^2+W_1\big(\pi_{t}^{N},\pi_{\theta(t,\Delta)}^{N,\Delta}\big)^2\Big] \nonumber\\
\leq & C\,\E^{\Q}\Big[\big|X_t^i-X_{t}^{i,\Delta}\big|^2+W_1\big(\pi_{t}^{N},\pi_{t}^{N,\Delta}\big)^2\Big] +C\Delta,
\end{align*}
where $f=b,\sigma,R,H$.
Then with the similar calculation that deduces Theorem \ref{multi}, if we take $\alpha$ large enough, then we have
\be \label{EMx}
\E^\Q\Big[\max\limits_{s\in[0,t]}\big|X_s^i-X_s^{i,\Delta}\big|^2\Psi_s^{\Delta}\Big]\leq C \Delta,
\ee
and 
\be\label{Deltabound}
\E^\Q\left[\max\limits_{s\in[0,t]}\left|X_s^{i}-X_s^{i,\Delta}\right|^2\right]\leq C\left|\ln\left(\Delta^{-1}\right)\right|^{-\frac{1}{2p}}.
\ee
Combining Theorem \ref{multi} and (\ref{Deltabound}), we can finish our proof.
\end{proof}

\section{Numerical Experiment}\label{secnum}
Although our theoretical error bound is $\mathcal{O}( [ \ln(N) ]^{-\frac{1}{2}+})$, the numerical scheme works very well in our numerical experiments (even in the cases when coefficients are not bounded).
Here we exhibit an example in which the solution can be written down explicitly.
Let us consider the following CMVSDE in the completed probability space $( \Omega, \mathbb{Q}, \mathcal{F})$:
\begin{equation}\label{Nume}
\left\{
\begin{aligned}
X_t & = 2x_0 + \int_0^t b\left( s, X_s, \pi_s \right) \D s+ \int_0^t \sigma_1 \left( s, X_s, \pi_s \right) \D W_s + \int_0^t \sigma_2 \left( s, X_s \right) \D Y_s, \\
L_t & = 1+ \int_0^t h\left( s, X_s, \pi_s \right) L_s\D Y_s,
\end{aligned}
\right.
\end{equation}
where $x_0\in\mathbb{R}$, and $(W,Y)$ is the standard $\mathbb{R}^2$-valued Brownian motion under $\mathbb{Q}$, and the mean-field term $\pi_s:=\mathscr{L}^\mathbb{P} (X_s |\mathcal{F}_s^Y)$ with $\frac{\D \mathbb{P}}{\D \mathbb{Q}}=L_T$, and the coefficients are defined by
$$
b(\omega,s,x,\pi) = b_0 \left[ x + \pi\left(\mathcal{L}\right) - \mathcal{R} \left(\omega,s\right) \right],
$$
$$
\sigma_1(\omega,s,x,\pi) = d_0 \left[ x + \pi\left(\mathcal{L}\right) - \mathcal{S} \left(\omega,s\right) \right],
$$
$$
\sigma_2(\omega,s,x) = c_0\,x_0\,e^{b_0s} \left[ x - \mathcal{T} \left(\omega,s\right) \right],
$$
$$
h(\omega,s,x,\pi) = c_0 \left[ x + \pi\left(\mathcal{L}\right) - \mathcal{U} \left(\omega,s\right) \right],
$$
with $( b_0, c_0, d_0 )\in\mathbb{R}^3$, and $\mathcal{L}$ being the linear function $\mathcal{L}(x) = x$ such that
$$
\pi_s\left(\mathcal{L}\right) = \E^\mathbb{P} \left[ \mathcal{L} \left(X_s\right) \Big|\mathcal{F}_s^Y\right] = \E^\mathbb{P} \left[ X_s \Big|\mathcal{F}_s^Y\right] = \frac{\E^\mathbb{Q} \big[ X_s L_s \big|\mathcal{F}_s^Y\big] }{ \E^\mathbb{Q} \big[ L_s \big|\mathcal{F}_s^Y\big]},
$$
and the stochastic processes being defined by
\begin{align*}\nonumber
\mathcal{R}\left(\omega,s\right) & = x_0 e^{b_0 s} \bigg[ 1+ \exp\Big(c_0 x_0 \int_0^s e^{br}\D Y_r \left(\omega\right) - \frac{1}{4b_0} c^2_0  x_0^2(e^{2b_0 s}-1)\Big) \bigg],\nonumber\\
\mathcal{S}\left(\omega,s\right) & = x_0 e^{b_0 s} \bigg[ 1 + 2\exp\Big( c_0 x_0 \int_0^s e^{br} \D Y_r \left(\omega\right) -\frac{1}{4b_0} c_0^2 x_0^2 (e^{2b_0 s}-1)\Big) \bigg],\nonumber\\
\mathcal{T}\left(\omega,s\right) & = x_0 \exp\left[ \left(b_0-\frac{1}{2} d_0^2 \right) s + d_0 W_s \left(\omega\right)\right],\nonumber\\
\mathcal{U}\left(\omega,s\right) & =  2 x_0 e^{b_0 s} \exp\left[ c_0 x_0 \int_0^s e^{b_0 r}\D Y_r \left(\omega\right)  -\frac{1}{4b_0} c^2_0 x_0^2 \Big(e^{2b_0 s}-1\Big) \right] \\
&\quad + x_0 \exp\bigg[ \Big(b_0 -\frac{1}{2} d_0^2 \Big) s+ d_0 W_s\left(\omega\right) \bigg].\nonumber
\end{align*}

It is straightforward to check that (\ref{Nume}) has a unique solution
$$X_t=x_0 \exp \left[ \Big( b_0 -\frac{1}{2}d_0^2\Big)t + d_0 W_t\right] + x_0 \exp\left[c_0 x_0 \int_0^t e^{b_0 s}\D Y_s - \frac{1}{4b_0} c_0^2 x_0^2 \left(e^{2b_0 t}-1 \right)+b_0 t\right].$$

Now we use the Euler scheme (\ref{Euler}) to compute the partially observable mean-field SDE (\ref{Nume}). In our numerical experiments, we take  $T=0.1$, $b_0=d_0=c_0=x_0=1$, and $\Delta=10^{-5}$. We sample N particles and repeat the simulation independently $N_{\text{compute}}=100$ times, where the particle number $N$ is taken from the set $\mathscr{S}:=\{10k-5:\,\,1\leq k\leq 10\}$. Finally, for each $N\in \mathscr{S}$, we record the error
$$
e_N:=\frac{1}{N}\max\limits_{1\leq n\leq T/\Delta}\bigg\{\sum\limits_{i=1}^N\left|X_{n\Delta}^i-X_{n\Delta}^{i,\Delta}\right|^2\bigg\}
$$
for $N_{\text{compute}}=100$ times and compute the mean of the error $e_N$. In Figure \ref{fig: numerical figure}, we give our computed results and we can observe that the algorithm has the $\mathcal{O}(\frac{1}{N})$ -convergence rate in this example.

\begin{figure}[H]
\centering
\includegraphics[scale=0.7]{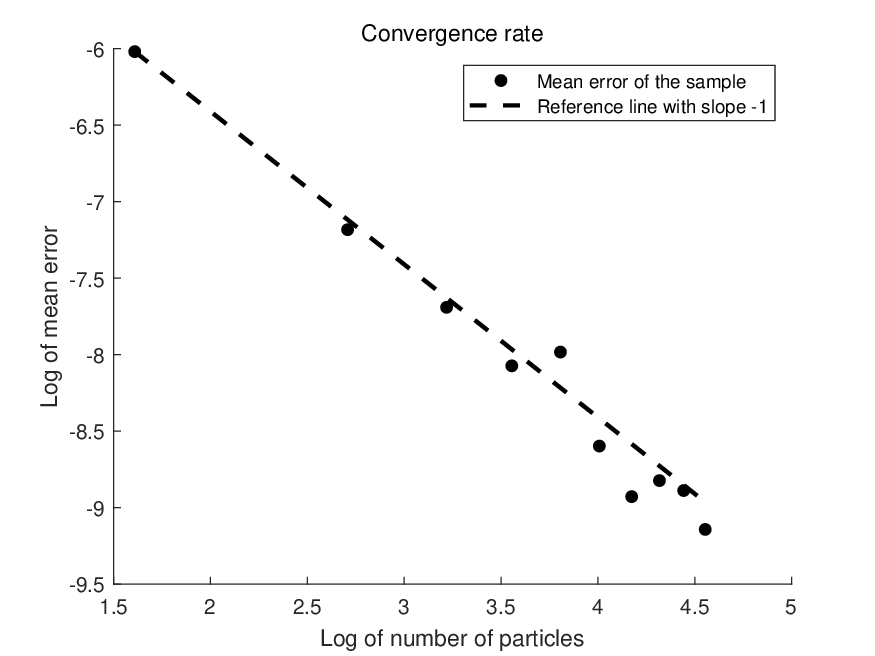}
\caption{Euler Scheme with Particle Approximation}
\label{fig: numerical figure}
\end{figure}

\bibliographystyle{plain}
\bibliography{ref}

\end{document}